\documentclass{amsart}
\usepackage{amsmath}
\usepackage{amssymb}
\usepackage{amsthm}
\usepackage{dsfont}
\usepackage{lineno}
\usepackage{subfigure}
\usepackage{graphicx}
\usepackage{multirow}
\usepackage{color}
\usepackage{xspace}
\usepackage{pdfsync}
\usepackage{hyperref} 
\usepackage{algorithmicx}
\usepackage{algpseudocode}
\usepackage{algorithm}
\usepackage{mathtools}
\usepackage{enumitem}
\usepackage{comment}
\graphicspath{{Figures/}}

\DeclareMathOperator*{\argmin}{argmin}

\newcommand{\bq}{\begin{equation}}
\newcommand{\eq}{\end{equation}}

\algnewcommand{\LineComment}[1]{\State \(\triangleright\) #1}

\newtheorem{theorem}{Theorem}

\theoremstyle{lemma}
\newtheorem{lemma}[theorem]{Lemma}

\newtheorem{corollary}[theorem]{Corollary}

\newtheorem{definition}[theorem]{Definition}

\newtheorem{hypothesis}[theorem]{Hypothesis}

\theoremstyle{remark}

\makeatletter
\newcommand\appendix@section[1]{%
\refstepcounter{section}%
\orig@section*{Appendix \@Alph\c@section: #1}%
}
\let\orig@section\section
\g@addto@macro\appendix{\let\section\appendix@section}
\makeatother

\begin{document}

\title[Preferential]{Optimal Transport Using Cost Functions with Preferential Direction with Applications to Optics Inverse Problems}

\author{Axel G. R. Turnquist}
\address{Department of Mathematics, University of Texas at Austin, Austin, TX, 78712}
\email{agrt@utexas.edu}

\begin{abstract}
We focus on Optimal Transport PDE on the unit sphere $\mathbb{S}^2$ with a particular type of cost function $c(x,y) = F(x \cdot y, x \cdot \hat{e}, y \cdot \hat{e})$ which we call cost functions with preferential direction, where $\hat{e} \in \mathbb{S}^2$. This type of cost function arises in an optics application which we call the point-to-point reflector problem. We define basic hypotheses on the cost functions with preferential direction that will allow for the Ma-Trudinger-Wang (MTW) conditions to hold and construct a regularity theory for such cost functions. For the point-to-point reflector problem, we show that the negative cost-sectional curvature condition does not hold. We will nevertheless prove the existence of a unique solution of the point-to-point reflector problem, up to a constant, provided that the source and target intensity are ``close enough''.
\end{abstract}

\date{\today}    
\maketitle

\section{Introduction}\label{sec:introduction}
Freeform optics is a treasure trove of interesting applications of Optimal Transport. An example of such an optics inverse problem that results in the Monge problem of Optimal Transport is the well-known reflector antenna problem, see~\cite{Wang_Reflector} and~\cite{Wang_Reflector2}. In the reflector antenna problem, light originates from the origin and has a given directional intensity. Then, using a reflector surface, the goal is to redirect the light to a desired far-field intensity. Solving for the Optimal Transport PDE with the cost function $c(x,y) = -\log(1-x \cdot y)$ on the unit sphere $\mathbb{S}^2$ can be used to find the shape of the freeform reflector antenna. 

In more generality, the Optimal Transport PDE is a particular formulation of the Monge problem of Optimal Transport, where the goal is to find a measure-preserving mapping which minimizes a cost functional by solving for a so-called ``potential function" (the solution of the PDE) instead of directly for the mapping. The Optimal Transport mapping then can be solved for in terms of the gradient of the potential function. In many applications, such as those in statistics (see~\cite{Peyre_ComputationalOT} for some more ``standard" applications of Optimal Transport), one needs to solve for the optimal mapping. In optics applications, however, the goal is instead to solve for the potential function, since the shape of the reflectors or lenses in the problem are related directly to the potential function.

Now, we briefly introduce the Optimal Transport problem. Given two subsets of Euclidean space $\Omega, \Omega ' \subset \mathbb{R}^d$, and two probability measure $\mu$ and $\nu$ having supports on $\Omega$ and $\Omega'$, respectively, we can ask the question whether or not there exists a mapping $T: \Omega \rightarrow \Omega'$ such that $T_{\#}\mu = \nu$. That is, for every Borel set $B \subset \Omega$, is it possible to find such a $T$ such that $\mu(B) = \nu(T(B))$. Under very general conditions~\cite{villani2}, it is possible to find such a $T$, in fact often many such $T$ exist. However, we can also formulate an optimization problem with the hopes of capturing a specific measure-preserving $T$. That is, we ask whether the further refinement of the problem has a solution. Given a ``cost" function $c: \Omega \times \Omega ' \rightarrow \mathbb{R}$, is there a $T$, such that:
\begin{equation}\label{eq:OT}
T = \argmin_{S_{\#}\mu = \nu} \int_{\Omega} c \left( x, S(x) \right) d\mu(x)?
\end{equation}

This is called the Monge problem of Optimal Transport, going back centuries, see Monge's work~\cite{monge}. The reader might be more familiar with the more general Kantorovich formulation of Optimal Transport, see~\cite{Villani1} or~\cite{santambrogio} for standard introductions.

It may seem like an unrealistic aspiration to find such a $T$. However, the classical results in~\cite{caffarelli} are that given some technical conditions on the measures $\mu$ and $\nu$, on the geometry of the set $\Omega '$ relative to the geometry of the set $\Omega$ and the cost function $c$, found in~\cite{MTW}, not only does there exist such a $T$, but under some circumstances $T$ can actually be found to be unique and $C^{\infty}$! So, not only is the $T$ measure-preserving, but we can find a particularly nice $T$. The regularity of $T$ depends on four things generally: (1) the smoothness of the source and target measures, (2) the geometry of the source and target sets (3) the underlying geometry of the space on which the measures lay, and (4) the cost function. What is interesting is that (1)-(3) could be as nice as you want, but still, the cost function might not allow for one to guarantee the smoothness of the Optimal Transport mapping.

Now, reining in our enthusiasm for a bit, we relax some of the technical conditions (just hoping $T$ to be $C^1$). We also assume $\mu$ and $\nu$ have density functions $f$ and $g$, respectively. Now not only is $T$ measure-preserving, but it was proved, originally in~\cite{brenier} and generalized later, that the minimizer $T$ of the functional $\mathcal{C}(S) = \int_{\Omega} c \left( x, S(x) \right) d\mu(x)$ is related to a potential function $u: \Omega \rightarrow \mathbb{R}$ which is $c$-convex, through the relation:
\begin{equation}\label{eq:relation}
\nabla u(x) = -\nabla_{x} c(x,y)
\end{equation}
for $y = T(x)$. Even though such a $T$ is unique, the potential function $u$ is unique only up to a constant. The definition of $c$-convexity generalizes that of convexity, and, as the nomenclature indicates, it depends on the particular cost function $c$ given in the Optimal Transport problem. In order to introduce the definition of $c$-convexity, we first introduce the $c$-transform.

\begin{definition}
Given a cost function $c:\Omega \times \Omega ' \rightarrow \mathbb{R}$, the $c$-transform of a function $u:\Omega \rightarrow \mathbb{R}$, which is denoted by $u^{c}$ is defined as:
\begin{equation}
    u^{c}(y) = \sup_{x \in \Omega} \left( -c(x,y) - u(x) \right)
\end{equation}
\end{definition}

Now that we have defined the $c$-transform, we can define what we mean by a function $u$ being $c$-convex.

\begin{definition}\label{def:cconvex}
A function $u$ is $c$-convex if at each point $x \in \Omega$, there exist $y \in \Omega '$ and a value $u^{c}(y)$ such that:
\begin{equation}
\begin{cases}
-u^{c}(y) - c(x,y) = u(x), \\
-u^{c}(y) - c(x', y) \leq u(x'), \ \forall x' \in \Omega
\end{cases}
\end{equation}
where $u^{c}(y)$ is the $c$-transform of $u$.
\end{definition}

From the measure-preserving property of $T$ and its relation to the potential function $u$, it was found~\cite{MTW} that the optimal mapping $T$ in Equation~\eqref{eq:OT} could be found via inverting Equation~\eqref{eq:relation} for $T$ and solving for the potential function $u$ in the following Optimal Transport PDE which is elliptic on the set of functions $u$ which are $c$-convex:
\begin{equation}\label{eq:OTPDE}
\det \left(D^2u(x) + D^2_{xx} c(x,y) \right) = \left\vert \det D^2_{xy} c(x,y) \right\vert f(x)/ g(T(x))
\end{equation}
for $y = T(x)$.

For the freeform optics problem we describe in this manuscript, the problem is not posed on subsets of Euclidean space $\mathbb{R}^d$, but instead on the unit sphere $\mathbb{S}^2 \subset \mathbb{R}^3$. We will assume, when our desire is to get regularity guarantees, that the sets $\Omega = \Omega ' = \mathbb{S}^2$, which means that $\mu$ and $\nu$ are positive measures on the whole sphere. Otherwise, the measures will simply have supports on the sphere. Formally, the equivalent Optimal Transport PDE formulation in Equation~\eqref{eq:OTPDE} and Equation~\eqref{eq:relation} on $\mathbb{S}^2$ look the same except that all differential operators are going to be with respect to the induced metric from $\mathbb{R}^3$ onto the unit sphere $\mathbb{S}^2$. 

One potential issue can be seen in that if the gradient of $u$ is too large in Equation~\eqref{eq:relation}, depending on the choice of cost function, the mapping $T$ may move ``too far", that is all the way to the cut locus from of the point $x \in \mathbb{S}^2$ (the antipodal point). This will cause issues with differentiability. In fact, this one of the main issues that was addressed, for the unit sphere, in~\cite{Loeper_OTonSphere}. For the cost functions $c(x,y) = \frac{1}{2}d_{\mathbb{S}^2}(x,y)^2$ and $c(x,y) = -\log(1 - x \cdot y)$ (as well as for a more general class of cost functions satisfying Theorem 4.1 in that paper), as long as the density functions $f$ and $g$ satisfied some technical conditions, then $T$ satisfied $d_{\mathbb{S}^2}\left( x, T(x)\right) \leq \pi - \delta$, for some $\delta>0$. In other words, the Optimal Transport mapping from~\eqref{eq:OT} would only transport mass a certain distance $\delta$ away from the cut locus (the antipodal point). With this guarantee, one could get the \textit{a priori} estimates on the potential function.

For some cost functions, not those treated in Theorem 4.1 of~\cite{Loeper_OTonSphere} mass cannot be transported beyond a certain distance, due to simply the structure of the cost function itself. In~\cite{T_LensRefractor}, for example, it was shown for the cost function $c(x,y) = -\log(n - x \cdot y)$ for $n>1$, beyond a certain distance the Optimal Transport mapping became complex-valued. This requirement was, of course, noted previously in~\cite{gutierrezhuang}. In~\cite{T_LensRefractor}, it was then possible, following the line of work by~\cite{Loeper_OTonSphere}, to check the MTW conditions for such cost functions and construct a regularity theory guaranteeing $C^{\infty}$ regularity of the mapping $T$ provided that the probability distributions $\mu$ and $\nu$ were positive $C^{\infty}$ measures that they did not require mass to move beyond what was allowed by the cost functions.


This manuscript can be thought of as a continuation of the work in~\cite{T_LensRefractor}. However, now we have a cost function that is significantly more complicated than in the previous work. Previously, all the cost functions treated could be written as a function of the dot product between $x$ and $y$, that is $c(x,y) = F(x\cdot y)$. This lead to the distance $d_{\mathbb{S}^2}(x,T(x))$ just depending on the magnitude of $\nabla u(x)$. That is, $T(x)$ was isotropic with respect to the direction of $\nabla u(x)$. In this manuscript, the mapping $T(x)$ arising from the Optimal Transport PDE we will encounter is anisotropic with respect to the direction of $\nabla u(x)$. This introduces many new complications to the formulas and the theory.

In this manuscript, we deal with cost functions that have a ``preferential direction", that is, the cost function $c(x,y) = F(x \cdot y, x \cdot \hat{e}, y \cdot \hat{e})$, where $\hat{e}$ is a particular fixed unit vector that arises from a given optical setup. This preferential direction $\hat{e}$ naturally arises in problems such as a point-to-point reflector problem, where directional light with intensity $f$ at a point $p_1$ reflects off two reflectors and focuses at another point $p_2$, where the directional intensity is $g$. The natural preferential direction is $\hat{e}$ the vector pointing from $p_1$ to $p_2$. No matter what choice of $\hat{e}$, whether one desires to have it point in the positive $z$ direction in $\mathbb{R}^3$, for example, does not obviate its complication on the cost function. 


In Section~\ref{sec:background} we introduce the point-to-point reflector problem from optics. We present the Optimal Transport PDE formulation and the resulting cost functions. We then introduce the Ma-Trudinger-Wang (MTW) conditions which are necessary to show for the existence of a unique solution, up to a constant, of a solution for the Optimal Transport PDE, as well as the necessary and sufficient condition to guarantee $C^{\infty}$ regularity of the solution of the Optimal Transport PDE given that the density functions are $C^{\infty}$ smooth and bounded away from zero. In Section~\ref{sec:computations}, we first introduce the hypotheses we will be making on the cost function. We then show that these hypotheses hold for the point-to-point cost function. We then prove that the hypotheses we assume show that the mixed Hessian is bounded away from zero. We finish by deriving the cost-sectional curvature condition for cost functions with preferential direction and showing that it does not hold for the point-to-point cost function. In Section~\ref{sec:regularity}, we prove a regularity theorem for cost functions satisfying Hypothesis~\ref{hyp:mainhypothesis}, the strictly negative cost-sectional curvature condition, and a condition on the source and target distributions satisfying some conditions that prevent mass from moving too ``far". We then show that this proves the existence and uniqueness of Lipschitz continuous solutions of the point-to-point problem with assuming mass does not move ``too far". In Section~\ref{sec:conclusion}, we summarize the results for cost functions with preferential directions and also specifically for the point-to-point cost function.

\section{Optics Problem and the MTW Conditions}\label{sec:background}
In this section, we introduce in more detail an optics problems that results in Optimal Transport PDE~\eqref{eq:OTPDE} with a cost function with preferential direction. The cost function we will be discussing arises from an optics inverse problem that involves designing the shape of reflectors used for reshaping light intensity patterns. We also introduce the Ma-Trudinger-Wang (MTW) conditions, originally introduced in~\cite{MTW}, but here we will use their statement from~\cite{Loeper_OTonSphere} as it pertains to the spherical geometry.

\subsection{Point-to-Point Reflector Problem}

It is perhaps well-known that a parabolic mirror can redirect parallel light to a single point, called the focus (of the parabolic mirror). In general, light emanating from a point $\mathcal{S}$ can be redirected via a reflector. It is possible, it turns out, to build two reflectors to take light emanating from a single point $\mathcal{S}$ and redirect it to pass through another single point $\mathcal{T}$. What is perhaps not obvious is that the light from $\mathcal{S}$ can be \textit{refocused} to $\mathcal{T}$ without attenuation using a surprising variety of reflectors. More precisely, for geometric optics (classical r\'{e}gime), we can redirect light from $\mathcal{S}$ to $\mathcal{T}$ and the light intensity is conserved. What is even less obvious is that it is possible to formulate a PDE whose solution (if it exists) can be used to solve for the shape of the reflectors used in this setup. And what is least obvious of all is that sometimes the PDE can be solved.

We desire to use two reflector surfaces to redirect light at a point $\mathcal{S}$ with directional intensity $f(x)$ to light at a point $\mathcal{T}$ with directional light intensity $g(y)$. In this notation, $x,y \in \mathbb{S}^2$ since they encode direction. We denote the two reflectors by the notation $\mathcal{R}_1$ and $\mathcal{R}_2$, where we choose $\mathcal{R}_1 = x\tilde{u}(x)$ and $\mathcal{R}_2 = y\tilde{v}(y)$. The direction from the source $\mathcal{S}$ the target $\mathcal{T}$ is denoted by the unit vector $\hat{e}$, and the distance separating the two points is denoted by $l$. The light at $\mathcal{S}$ with radial intensity $f(x)$ travels in the direction $x$, reflects off $\mathcal{R}_1$, then reflects off $\mathcal{R}_2$, finally traveling in the direction $y$ through $\mathcal{T}$ with radial intensity $g(y)$, see Figure~\ref{fig:p2p}. Interestingly, the length of the optical path from $\mathcal{S}$, reflecting off $\mathcal{R}_1$, then reflecting off $\mathcal{R}_2$ and reaching the point $\mathcal{T}$ is a fixed constant, which we denote by $L$.

\begin{figure}[htp]
	\centering
	\includegraphics[width=0.8\textwidth]{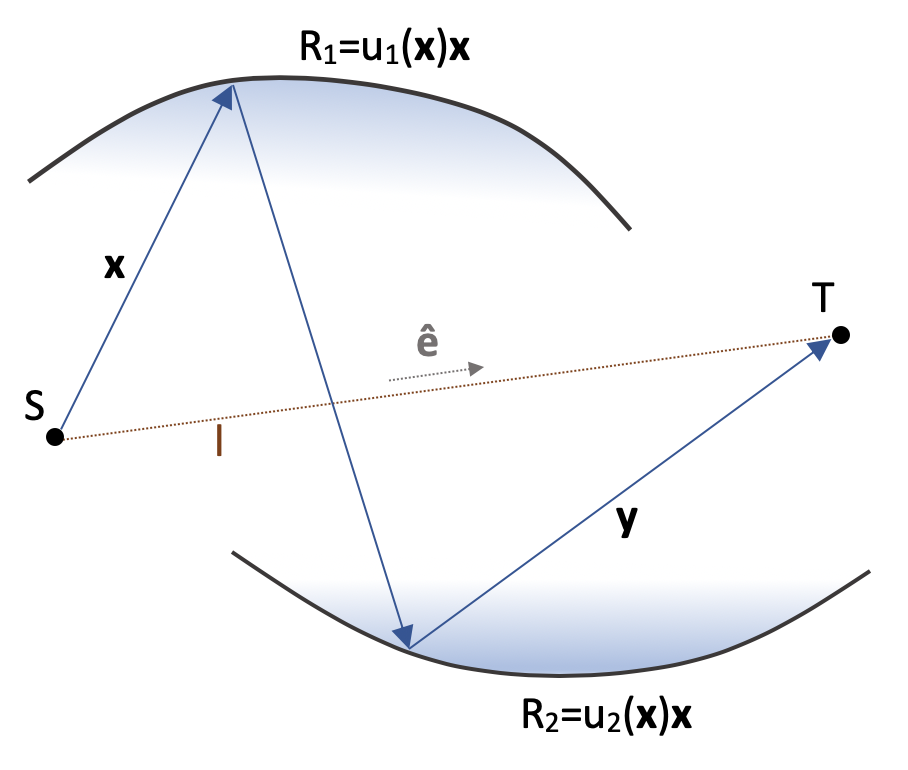}
	\caption{Light from the source $\mathcal{S}$ with intensity $f(x)$ in the direction $x$ reflects of the reflector $\mathcal{R}_1$ and then the reflector $\mathcal{R}_2$, traveling in the direction $y$, where the resulting directional intensity pattern is $g(y)$ at the target point $\mathcal{T}$.}
	\label{fig:p2p}
\end{figure}

\subsubsection{PDE Formulation}

Following the derivation in~\cite{YadavThesis}, we assume conservation of light intensity, that is, for any subset $A \subset \mathbb{S}^2$ where $T(A) \subset \mathbb{S}^2$, we have
\begin{equation}
\int_{A} f(x) dx = \int_{T(A)} g(y) dy.
\end{equation}

Then, given this conservation of light intensity, the shapes of the reflectors $\mathcal{R}_1 = \tilde{u}(x)$ and $\mathcal{R}_2 = \tilde{v}(y)$, the laws of reflection, and the change of variables from $\tilde{u}$ to $u$ and $\tilde{v}$ to $v$ via the equations:
\begin{equation}
\tilde{u}(x) = \frac{L^2-l^2}{\left( L-l(x \cdot \hat{e}) \right) \left( (L^2-l^2) e^{u(x)} + 1\right)}
\end{equation}

and
\begin{equation}
\tilde{v}(x) = \frac{L^2-l^2}{\left( L-l(y \cdot \hat{e}) \right) \left( (L^2-l^2) e^{v(x)} + 1\right)}
\end{equation}
we arrive at the PDE:
\begin{equation}
\det \left( D^2 u(x) + D^2_{xx} c(x,y) \vert_{y = T(x)} \right) = \frac{\left\vert D^2_{xy} c(x,y) \vert_{y = T(x)} \right\vert f(x)}{g(T(x))}
\end{equation}
where $T(x)$ satisfies Equation~\eqref{eq:relation} and $u(x) + v(T(x)) = c(x,T(x))$. According to the derivation in Section 3.4 of~\cite{YadavThesis}, the cost function for the point-to-point system is:
\begin{equation}
c(x,y) = \log \left( \frac{1}{(L^2-l^2)^2} - \frac{1-x \cdot y}{2(L^2-l^2)(L-lx \cdot \hat{e})(L-l y \cdot \hat{e})} \right).
\end{equation}

\subsection{Ma-Trudinger-Wang Conditions for Defective Cost Functions}\label{sec:MTW}

The MTW conditions were originally formulated in~\cite{MTW}, but we focus on the Riemannian generalization as stated in~\cite{Loeper_OTonSphere}. Given a compact domain $D \subset \mathbb{S}^2 \times \mathbb{S}^2$, denote by $\pi_1: \mathbb{S}^2 \times \mathbb{S}^2 \mapsto \mathbb{S}^2$, the projection $\pi_1(x, y) = x$ and its inverse $\pi_{1}^{-1}(x) = \left\{ x \right\} \times \mathbb{S}^2$. For any $x\in \pi_{1}(D)$, we denote by $D_{x}$ the set $D \cap \pi_{1}^{-1}(x)$. Then, we introduce the following conditions:

\begin{hypothesis}
\begin{itemize}
\item [$\mathbf{A0}$] The cost function $c$ belongs to $C^{4}(D)$.
\item [$\mathbf{A1}$] For all $x \in \pi_{1}(D)$, the map $y \rightarrow -\nabla_{x}c(x,y)$ is injective on $D_{x}$.
\item [$\mathbf{A2}$] The cost function $c$ satisfies $\det D^{2}_{xy}c \neq 0$ for all $(x,y)$ in $D$.
\item [$\mathbf{Aw}$] The cost-sectional curvature is non-negative on $D$. That is, for all $(x,y) \in D$, for all $\xi, \eta \in T_{x}\mathbb{S}^2$, $\xi \perp \eta$, 
\begin{equation}
\mathfrak{G}_{c}(x, y)(\xi, \eta) \geq 0
\end{equation}
\item [$\mathbf{As}$] The cost-sectional curvature is uniformly positive on $D$. That is, for all $(x,y) \in D$, for all $\xi, \eta \in T_{x}\mathbb{S}^2$, $\xi \perp \eta$, 
\begin{equation}
\mathfrak{G}_{c}(x, y)(\xi, \eta) \geq C_0 \left\vert \xi \right\vert^2 \left\vert \eta \right\vert^2
\end{equation}
\end{itemize}
\end{hypothesis}

We will often colloquially state that the MTW conditions ``hold" for a particular cost function. When it is necessary to be more precise, we will state specifically for which domain $D \subset \mathbb{S}^2 \times \mathbb{S}^2$ these specific conditions hold.

Our aim in this paper is to verify the MTW conditions $\mathbf{A1}$ and $\mathbf{A2}$ for the point-to-point and point-to-plane cost functions on a certain domain $D$. We will additionally establish a general regularity theory regarding cost functions with preferential direction that satisfy $\mathbf{A0}$, $\mathbf{A1}$, $\mathbf{A2}$ and $\mathbf{As}$. Establishing that $\mathbf{A0}$, $\mathbf{A1}$ and $\mathbf{A2}$ hold on some domain $D$, which will be shown to be achievable with a wide class of density functions, guarantees the existence (up to a constant) of a solution of the Optimal Transport PDE and also guarantees the smoothness of solutions given the smoothness of the source and target measures and also that their supports are on sets $\Omega$ and $\Omega '$ which satisfy a geometric condition, see~\cite{MTW}. For the sphere, if we desire smoothness, we assume the measures are positive over the entire sphere and thus bypass the possible difficulties with the geometries of $\Omega$ and $\Omega '$ that are tricky in Euclidean space. It turns out that if $\mathbf{A0}$, $\mathbf{A1}$, $\mathbf{A2}$ hold, but $\mathbf{Aw}$ does not, then there exist $C^{\infty}$ density functions $f$ and $g$, bounded away from zero, such that the Optimal Transport mapping $T$ from Equation~\eqref{eq:OT} is not even guaranteed to be continuous, see~\cite{LoeperReg}.

We will show in Section~\ref{sec:computations} that $\mathbf{A0}$ and $\mathbf{A1}$ hold by important hypotheses that we will state. Given these hypotheses, we will show that $\mathbf{A2}$ holds as well. We will show, however, that $\mathbf{Aw}$ does not hold on any domain $D$ for the point-to-point cost function.

\section{Computations}\label{sec:computations}
\subsection{Cost Functions with Preferential Directions}

In this subsection, we introduce the concept of a cost function with preferential direction and demonstrate that this definition can be applied to the point-to-point optics problem. We simplify the formula of the cost function, which will be helpful in later formulas. 

\begin{definition}
A cost function $c(x,y): \mathbb{S}^2 \times \mathbb{S}^2 \rightarrow \mathbb{R}$ that is symmetric in its arguments will be called a cost function with preferential direction if it can be written as $c(x,y) = F(x \cdot y, x \cdot \hat{e}, y \cdot \hat{e})$ for some fixed unit vector $\hat{e} \in \mathbb{S}^2$.
\end{definition}

First, we simplify the formula for the cost function for the point-to-point reflector problem, which clearly is a cost function with preferential direction:
\begin{equation}
c(x,y) = \log \left( \frac{1}{(L^2-l^2)^2} - \frac{1-x \cdot y}{2(L^2-l^2)(L-lx \cdot \hat{e})(L-l y \cdot \hat{e})} \right) = F(x \cdot y, x \cdot \hat{e}, y \cdot \hat{e}).
\end{equation}

We can rewrite the cost function as follows:
\begin{align}
c(x,y) &= \log \left( \frac{1}{(L^2-l^2)^2} \left(1 - \frac{(L^2-l^2)(1-x \cdot y)}{2(L-lx \cdot \hat{e})(L-l y \cdot \hat{e})} \right) \right), \\
&= \log \left( \frac{1}{(L^2-l^2)^2} \right) + \log \left(1 - \frac{(L^2-l^2)(1-x \cdot y)}{2(L-lx \cdot \hat{e})(L-l y \cdot \hat{e})} \right),
\end{align}
and since the Optimal Transport PDE only depends on derivatives of the cost function, we may thus translate the cost function by a constant and use the following cost function for the point-to-point problem:
\begin{equation}
c(x,y) = \log \left( 1 - \frac{(L^2-l^2)(1-x \cdot y)}{2(L-lx \cdot \hat{e})(L-l y \cdot \hat{e})} \right) = F(x \cdot y, x \cdot \hat{e}, y \cdot \hat{e}).
\end{equation}

We denote $a = l/L$. Notice that $0<a<1$. With this change of notation, the cost function becomes:
\begin{equation}\label{eq:cost1}
c(x,y) = \log \left( 1 - \frac{(1-a^2)(1-x \cdot y)}{2(1-ax \cdot \hat{e})(1-a y \cdot \hat{e})} \right).
\end{equation}
We will refer to the cost function in Equation~\eqref{eq:cost1} as the point-to-point cost function.

\subsection{Solving for the Mapping}\label{sec:dotproduct}

In this subsection, we will be providing some basic discussion about the hypotheses we wish to impose on the cost function with preferential direction, then we will list the hypotheses explicitly.

We rewrite Equation~\eqref{eq:relation} as
\begin{equation}
p = -\nabla_{x} c(x,y),
\end{equation}
where $p$ denotes an element of the tangent plane at $x$. For cost functions of the type
\begin{equation}
c(x,y) = F(x \cdot y, x \cdot \hat{e}, y \cdot \hat{e}),
\end{equation}
we would like for there to be a one-to-one correspondence between $p$ and $y$. Given $x$ and $p$, define $q = x \times p$. Then, Equation~\eqref{eq:relation}, supplemented with constraints on $y$, can be written as the four following equations:
\begin{align}\label{eq:system1}
0 &= \left\Vert p \right\Vert + F_1(x \cdot y, x \cdot \hat{e}, y \cdot \hat{e}) y \cdot \hat{p} + F_2(x \cdot y, x \cdot \hat{e}, y \cdot \hat{e}) \hat{e} \cdot \hat{p} \\
0 &= F_1(x \cdot y, x \cdot \hat{e}, y \cdot \hat{e}) y \cdot \hat{q} + F_2(x \cdot y, x \cdot \hat{e}, y \cdot \hat{e}) \hat{e} \cdot \hat{q} \\
0 &= (x \cdot y)^2 +(y \cdot \hat{p})^2 + (y \cdot \hat{q})^2-1 \\
0 &= (x \cdot y)x \cdot \hat{e} + (y \cdot \hat{p}) \hat{e} \cdot \hat{p} + (y \cdot \hat{q}) \hat{e} \cdot \hat{q} - y \cdot \hat{e}.
\end{align}

We will need to be able to solve these equations for the variables $\xi_1 = (x \cdot y, y \cdot \hat{p}, y \cdot \hat{q}, y \cdot \hat{e})$ in terms of the variables $\xi_2 = (\left\Vert p \right\Vert, x \cdot \hat{e}, \hat{e} \cdot \hat{p}, \hat{e} \cdot \hat{q})$. Denote Equation~\eqref{eq:system1} by $0 = H(\xi_1, \xi_2)$, that is,
\begin{multline}\label{eq:Hdefinition}
H(\xi_1, \xi_2) = H(x \cdot y, y \cdot \hat{p}, y \cdot \hat{q}, y \cdot \hat{e}, \left\Vert p \right\Vert, x \cdot \hat{e}, \hat{e} \cdot \hat{p}, \hat{e} \cdot \hat{q}) = \\
\begin{bmatrix}
\left\Vert p \right\Vert + F_1(x \cdot y, x \cdot \hat{e}, y \cdot \hat{e}) y \cdot \hat{p} + F_2(x \cdot y, x \cdot \hat{e}, y \cdot \hat{e}) \hat{e} \cdot \hat{p} \\
F_1(x \cdot y, x \cdot \hat{e}, y \cdot \hat{e}) y \cdot \hat{q} + F_2(x \cdot y, x \cdot \hat{e}, y \cdot \hat{e}) \hat{e} \cdot \hat{q} \\
(x \cdot y)^2 +(y \cdot \hat{p})^2 + (y \cdot \hat{q})^2-1 \\
(x \cdot y)x \cdot \hat{e} + (y \cdot \hat{p}) \hat{e} \cdot \hat{p} + (y \cdot \hat{q}) \hat{e} \cdot \hat{q} - y \cdot \hat{e}
\end{bmatrix}.
\end{multline}

Suppose that $\xi_2$ is fixed on some domain $\Omega_2$. We need to assume that Equation~\eqref{eq:system1} can be solved for $\xi_1$. This will allow us to, for a fixed $x$ to solve for $y$ in terms of $p$. 

In order to characterize the smooth dependence of $x \cdot y$, $y \cdot \hat{p}$, $y \cdot \hat{q}$ and $y \cdot \hat{e}$ on the variables $\left\Vert p \right\Vert$, $x \cdot \hat{e}$, $\hat{e} \cdot \hat{p}$ and $\hat{e} \cdot \hat{q}$, the most convenient tool is the implicit function theorem. On the other hand, directly solving for $\xi_1$ and showing differentiability is another alternative. This is what we will be doing in Section~\ref{sec:verify}. However, generally, the implicit function theorem gives us a sufficient condition for what we need. We require the following:
\begin{equation}
\left\vert D_{\xi_{1}}H \right\vert = 
\begin{vmatrix}
F_{11} y \cdot \hat{p} + F_{12} \hat{e} \cdot \hat{p} & F_1 & 0 & F_{13} y \cdot \hat{p} + F_{23} \hat{e} \cdot \hat{p} \\
F_{11} y \cdot \hat{q} + F_{12} \hat{e} \cdot \hat{q} & 0 & F_1 & F_{13} y \cdot \hat{q} + F_{23} \hat{e} \cdot \hat{q} \\
2 x \cdot y & 2 y \cdot \hat{p} & 2 y \cdot \hat{q} & 0 \\
x \cdot \hat{e} & \hat{e} \cdot \hat{p} & \hat{e} \cdot \hat{q} & -1
\end{vmatrix} \neq 0.
\end{equation}

This will then imply that we can solve uniquely for the variables $\xi_1$ in terms of the variables $\xi_2$. That is, there exists a function $K$ such that:
\begin{equation}\label{eq:HK}
0 = H(K(\xi_2), \xi_2).
\end{equation}

Furthermore, it also implies that each component of $K$ is differentiable in terms of the variable $\xi_1$:
\begin{equation}
\left[ \frac{\partial K}{\partial \xi_2} \right] = -\left[ D_{\xi_{1}} H(K(\xi_2), \xi_2) \right]^{-1} \left[ D_{\xi_{2}} H(K(\xi_2), \xi_2) \right],
\end{equation}
where,
\begin{equation}
D_{\xi_{2}} H(\xi_1, \xi_2) = \begin{pmatrix}
1 & F_{12} y \cdot \hat{p} + F_{22} \hat{e} \cdot \hat{p} & F_2 & 0 \\
0 & F_{12} y \cdot \hat{q} + F_{22} \hat{e} \cdot \hat{q} & 0 & F_2 \\
0 & 0 & 0 & 0 \\
0 & x \cdot y & y \cdot \hat{p} & y \cdot \hat{q} 
\end{pmatrix}.
\end{equation}

In order to bound each derivative $\frac{\partial K_i}{\partial (\xi_2)_{j}}$, we will need bounds on the functions $F_1, F_2, F_{11}, F_{12}, F_{13}, F_{22}$, and $F_{23}$.

Now, suppose given $x$ and $\hat{e}$, we are given a resulting $y$. We need it to be possible to know the value of $p$ that yielded $y$. This is assumption $\mathbf{A1}$ in the MTW conditions. Hence, we need to assume the system in Equation~\eqref{eq:system1} can also be solved with respect to the variables $\left\Vert p \right\Vert, \hat{e} \cdot \hat{p}$, and $\hat{e} \cdot \hat{q}$. We will additionally want to assume that:
\begin{equation}
\left\vert \det D_{\xi_2; x \cdot \hat{e}} H \right\vert = \begin{vmatrix}
1 & F_2 & 0 \\
0 & 0 & F_2 \\
0 & y \cdot \hat{p} & y \cdot \hat{q}
\end{vmatrix} = -F_2 y \cdot \hat{p} \neq 0.
\end{equation}

These assumptions mean that, on a relevant domain, there is a one-to-one, differentiable correspondence between $y$ and $p$.

We also find the conditions on $c(x,y)$ needed to guarantee that $x \cdot y \rightarrow 1$ as $\left\Vert p \right\Vert \rightarrow 0$.

\begin{lemma}\label{lemma:startatzero}
Assuming $F_2(1, x \cdot \hat{e}, y \cdot \hat{e})=0$, but $F_2(x \cdot y, x \cdot \hat{e}, y \cdot \hat{e}) \neq 0$ for all $x \cdot y \neq 1$ and $F_1 \neq 0$, then $x \cdot y = 1$ as $\left\Vert p \right\Vert \rightarrow 0$.
\end{lemma}

\begin{proof}

Equation~\eqref{eq:system1} gives us:
\begin{align}
0 = \left\Vert p \right\Vert + F_1 y \cdot \hat{p} + F_2 \hat{e} \cdot \hat{p} \\
0 = F_1 y \cdot \hat{q} + F_2 \hat{e} \cdot \hat{q}.
\end{align}

Then, as we take the limit $\left\Vert p \right\Vert \rightarrow 0$, we get the system:
\begin{equation}
\begin{pmatrix}
y \cdot \hat{p} & \hat{e} \cdot \hat{p} \\
y \cdot \hat{q} & \hat{e} \cdot \hat{q}
\end{pmatrix}
\begin{pmatrix}
F_1 \\
F_2
\end{pmatrix} = \begin{pmatrix}
0 \\
0
\end{pmatrix}
\end{equation}

Let us first look at the case where $x \neq \pm \hat{e}$ and $y \neq \pm x$. Since $x \neq \pm \hat{e}$ and $y \neq \pm x$, the matrix is invertible, so the solution should be $F_1 = F_2 = 0$, which is not possible by the assumptions of the theorem. If we allow $x = \pm \hat{e}$, but still $y \neq \pm x$, we get:
\begin{equation}
\begin{pmatrix}
y \cdot \hat{p} & 0 \\
y \cdot \hat{q} & 0
\end{pmatrix}
\begin{pmatrix}
F_1 \\
F_2
\end{pmatrix} = 
\begin{pmatrix}
y \cdot \hat{p} F_1 \\
y \cdot \hat{q} F_1
\end{pmatrix}=
\begin{pmatrix}
0 \\
0
\end{pmatrix},
\end{equation}
which is also not possible, since $y \neq \pm x$ and $F_1 \neq 0$. If $y = \pm x$, but $x \neq \pm \hat{e}$, then this would require that $F_2 = 0$. Since $F_2 = 0$ only when $y = x$, then we cannot have $y = -x$ when $\left\Vert p \right\Vert \rightarrow 0$. By the hypotheses on the cost function, $y$ is continuous in $x \cdot \hat{e}$. Therefore, by continuity, we have $y = x$ as $\left\Vert p \right\Vert \rightarrow 0$ when $x = \pm \hat{e}$ as well.

\end{proof}

Note that a similar argument can be made for $x \cdot y \rightarrow -1$, with the conditions on $F_2$ changed \textit{mutatis mutandis}.

For clarity of exposition, we now state the assumptions that we have informally discussed on the cost function $c(x,y) = F(x \cdot y, x \cdot \hat{e}, y \cdot \hat{e})$, as well as some new assumptions and their implications. These assumptions can be compared with and their analogous (albeit simpler) definitions for defective cost functions in~\cite{T_LensRefractor}.

\begin{hypothesis}\label{hyp:mainhypothesis}
Define $H$ by Equation~\eqref{eq:Hdefinition}. We make the following assumptions on the cost function $c(x,y): \mathbb{S}^2 \times \mathbb{S}^2 \rightarrow \mathbb{R}$, where $c(x,y) = F(x \cdot y, x \cdot \hat{e}, y \cdot \hat{e})$ for $\hat{e} \in \mathbb{S}^2$ fixed:
\begin{enumerate}
\item Let $H(\xi_1, \xi_2)$ be solvable for $\xi_1$ for all $\xi_2 \in \Omega_2$, where $\Omega_2 \subset \mathbb{R} \times [-1,1]^3$ and where $[0, p^{*}] \times [-1,1]^3 \subset \Omega_2$ for some $p^{*}>0$. This will allow us to solve for $y$ once we have $p$.
\item Let $H(\xi_1, \xi_2)$ be solvable for $\xi_2$ for all $\xi_1 \in \Omega_1$, where $\Omega_1 \subset [-1,1]^4$ and where $[\xi^{*}, 1] \times [-1,1]^3 \subset \Omega_1$ for some $\xi^{*}<1$. This allows us to verify assumption \textbf{A1} of the MTW conditions. With this additional assumption, we have a one-to-one correspondence between $y$ and $p$. Note, that we could instead make the assumption $[-1, \xi^{*}] \times [-1,1]^3 \subset \Omega_1$ for some $\xi^{*}>-1$.
\item Let $\left\vert D_{\xi_{1}} H \right\vert \neq 0$ on $\Omega_1 \times \Omega_2$ or, via directly solving, show that $\xi_1$ has smooth dependence on $\xi_2$ on $\Omega_1 \times \Omega_2$. This will allow us to smoothly solve for $y$ in terms of $x$, $\hat{e}$ and $p$ and help show that the mixed Hessian term is bounded away from zero, which is assumption \textbf{A2} of the MTW conditions.
\item Let $\left\vert D_{\xi_2; x \cdot \hat{e}} H \right\vert \neq 0$ on $\Omega_1 \times \Omega_2$. This will allow us to smoothly solve for $p$ in terms of $x$, $\hat{e}$ and $y$ and help show that the mixed Hessian term is bounded away from zero, which is assumption \textbf{A2} of the MTW conditions.
\item Let $F_1, F_2, F_{11}, F_{12}, F_{13}, F_{22}, F_{23}$ be bounded on $\Omega \subset [-1,1]^3$ and where $[\xi^{*}, 1] \times [-1,1]^2 \subset \Omega$ for some $\xi^{*}<1$. This will help show that the mixed Hessian term is bounded away from zero, which is assumption \textbf{A2} of the MTW conditions. Note, that we could instead make the assumption $[-1, \xi^{*}] \times [-1,1]^2 \subset \Omega$ for some $\xi^{*}>-1$.
\item Let $F_1 \neq 0$ on $\Omega$. This will help show that the mixed Hessian term is bounded away from zero, which is assumption \textbf{A2} of the MTW conditions. This is also important in showing that $y \rightarrow x$ as $\left\Vert p \right\Vert \rightarrow 0$.
\item Let $F_2/\left\Vert p \right\Vert$ be bounded as $\left\Vert p \right\Vert \rightarrow 0$ on $\Omega$. This will help show that the mixed Hessian term is bounded away from zero, which is assumption \textbf{A2} of the MTW conditions.
\item Let $F_2(1, x \cdot \hat{e}, y \cdot \hat{e}) = 0$ and $F_2(x \cdot y, x \cdot \hat{e}, y \cdot \hat{e}) \neq 0 \ \forall x \cdot y \neq 1$ onn $\Omega$. This assumption shows that $\lim_{\left\Vert p \right\Vert \rightarrow 0} y = x$ for the point-to-point cost function. The same assumption is not necessary to make for the point-to-plane cost function, where $\lim_{\left\Vert p \right\Vert \rightarrow 0} y = -x$, due to the simplicity of the structure of the cost function. However, for general cost functions where $\lim_{\left\Vert p \right\Vert \rightarrow 0} y = -x$, a similar assumption would need to be made.
\item Let $c(x,y) = F(x \cdot y, x \cdot \hat{e}, y \cdot \hat{e})$ be $C^{4}$ on $\Omega$. This is the MTW condition $\mathbf{A0}$. This will be unnecessary to verify for the point-to-point cost function, since this condition is only needed to check the cost-sectional curvature condition $\mathbf{Aw}$ and $\mathbf{As}$, which the point-to-point cost function fails, and the \textit{a priori} estimates, which we cannot use since solutions to the Optimal Transport PDE with the point-to-point cost function cannot even be guaranteed to be continuous.
\end{enumerate}
\end{hypothesis}

\subsection{Verifying Properties of the Point-to-Point Cost Function}\label{sec:verify}

\subsubsection{Identifying the Sets $\Omega_1$, $\Omega_2$, and $\Omega$ for the Point-to-Point Cost Function}

We begin by identifying the sets $\Omega_1, \Omega_2$ and $\Omega$ in Hypothesis~\ref{hyp:mainhypothesis} for the point-to-point cost functions.

For the point-to-point cost function, the easiest way to identify $\Omega_2$ is to simply solve for the mapping directly. Using Equation~\eqref{eq:system1}, we get:
\begin{equation}
-\left\Vert p \right\Vert \left( \alpha(1-a x \cdot \hat{e})(1 - a y \cdot \hat{e})-1 + x \cdot y \right) = y \cdot \hat{p} - \frac{a \hat{e} \cdot \hat{p} (1 - x \cdot y)}{1 - a x \cdot \hat{e}},
\end{equation}
where $\alpha = 2/(1-a^2)$. We then use $y \cdot \hat{e} = (x \cdot y)x \cdot \hat{e} + (y \cdot \hat{p}) \hat{e} \cdot \hat{p} + (y \cdot \hat{q}) \hat{e} \cdot \hat{q}$ and get:
\begin{multline}
-\left\Vert p \right\Vert \left( \alpha (1 - a x \cdot \hat{e})-1 + x \cdot y - a \alpha(1-a x \cdot \hat{e})((x \cdot y) x \cdot \hat{e} + (y \cdot \hat{p}) \hat{e} \cdot \hat{p} + (y \cdot \hat{q}) \hat{e} \cdot \hat{q}) \right) = \\
y \cdot \hat{p} - \frac{a \hat{e} \cdot \hat{p} (1 - x \cdot y)}{1 - a x \cdot \hat{e}}.
\end{multline}

Now, we use $y \cdot \hat{q} = \frac{a(\hat{e} \cdot \hat{q})(1-x \cdot y)}{1 - a x \cdot \hat{e}} =  \frac{a \hat{e} \cdot \hat{q}}{\beta}(1-x \cdot y) = \alpha_2 + \beta_2(x \cdot y)$, where $\beta = 1-a x \cdot \hat{e}$. Rewriting this, we get:
\begin{multline}
\left\Vert p \right\Vert \left(1 - \alpha \beta + a^2 \alpha (\hat{e} \cdot \hat{q})^2 \right) + \frac{a \hat{e} \cdot \hat{p})}{\beta} + \\
(x \cdot y) \left( \left\Vert p \right\Vert \left( a \alpha \beta x \cdot \hat{e} - 1 - a^2 \alpha (\hat{e} \cdot \hat{q})^2 \right) - \frac{a (\hat{e} \cdot \hat{p})}{\beta} \right) = (y \cdot \hat{p}) \left( 1 - \left\Vert p \right\Vert a \alpha \beta (\hat{e} \cdot \hat{p}) \right).
\end{multline}

Relabeling $\gamma_1 = 1 + a^2\alpha (\hat{e} \cdot \hat{q})^2$ and $\gamma_2 = \frac{a (\hat{e} \cdot \hat{p})}{\beta}$. Then, we get:
\begin{equation}\label{eq:ypyx}
y \cdot \hat{p} = \frac{\left\Vert p \right\Vert (\alpha \beta + \gamma_1)+\gamma_2}{1 - \left\Vert p \right\Vert a \alpha \beta (\hat{e} \cdot \hat{p})}+ (x \cdot y) \frac{\left\Vert p \right\Vert (a \alpha \beta x \cdot \hat{e} - \gamma_1)-\gamma_2}{1 - \left\Vert p \right\Vert a \alpha \beta (\hat{e} \cdot \hat{p})} = \alpha_1 + \beta_1 (x \cdot y).
\end{equation}

We thus identify values of $\left\Vert p \right\Vert$ where $\alpha_1, \alpha_2, \beta_1, \beta_2$ are bounded. Denote the largest $\left\Vert p \right\Vert$ that satisfies $\left\Vert p \right\Vert< \frac{1-a^2}{2a(1-a x \cdot \hat{e})}$ as $p_1$. We can now solve for $K$ such that $0 = H(K(\xi_2), \xi_1)$. Since we have Equation~\eqref{eq:ypyx}, we can then use $(y \cdot x)^2 + (y \cdot \hat{p})^2 + (y \cdot \hat{q})^2 = 1$ and solve for $ x \cdot y$ which yields:
\begin{equation}
(x \cdot y)^2 \left( 1 + \beta_1^2 + \beta_2^2 \right) + 2(x \cdot y) \left( \alpha_1 \beta_1 + \alpha_2 \beta_2 \right) + \left(\alpha_1^2 + \alpha_2^2-1 \right) = 0
\end{equation}
and thus:
\begin{equation}\label{eq:xdoty}
x \cdot y = \frac{-\alpha_1 \beta_1 - \alpha_2 \beta_2 \pm \sqrt{(\alpha_1 \beta_1 + \alpha_2 \beta_2)^2 - (1+\beta_1^2+\beta_2^2)(\alpha_1^2+\alpha_2^2-1)}}{1+ \beta_1^2 + \beta_2^2}.
\end{equation}

We check when $x \cdot \hat{e} = 1$, $\hat{e} \cdot \hat{p} = \hat{e} \cdot \hat{q} = 0$, as we then take $\left\Vert p \right\Vert \rightarrow 0$, we expect to recover $x \cdot y = 1$, and thus we need to take the positive sign. Therefore,
\begin{equation}\label{eq:r1expression}
x \cdot y = \frac{-\alpha_1 \beta_1 - \alpha_2 \beta_2 + \sqrt{(\alpha_1 \beta_1 + \alpha_2 \beta_2)^2 - (1+\beta_1^2+\beta_2^2)(\alpha_1^2+\alpha_2^2-1)}}{1+ \beta_1^2 + \beta_2^2}.
\end{equation}
Then, using this, we can solve for $y \cdot \hat{p} = \alpha_1 + \beta_2 (x \cdot y)$ and $y \cdot \hat{q} = \alpha_2 + \beta_2 (x \cdot y)$. This gives us the function $K$. We identify when the discriminant is positive.

As $\left\Vert p \right\Vert \rightarrow 0$, we compute the discriminant:
\begin{equation}
\frac{2a^4}{\beta^4}(\hat{e} \cdot \hat{p})^2 \left((\hat{e} \cdot \hat{p})^2 +(\hat{e} \cdot \hat{q})^2 \right) + 1 \geq 1 >0.
\end{equation}

The discriminant is either a smooth function of $\left\Vert p \right\Vert$ up to when it first equals zero, which we denote $p_2$, or the discriminant is positive up to $p_1$. We denote $\tilde{p}$ as equal to $p_2$ if it exists, or $p_1$ if $p_2$ does not exist. We then define:
\begin{equation}
\tilde{\Omega}_2 = [0, \tilde{p}] \times [-1,1]^3,
\end{equation}
and, then, $\tilde{\Omega}_1 = K(\tilde{\Omega}_2)$. Notice that $K$ is a differentiable function on $\tilde{\Omega}_2$. Since $x \cdot y = 1 + \mathcal{O}\left( \left\Vert p \right\Vert \right)$, we denote $\tilde{\xi}^{*} \neq 1$ as the largest value such that $[1,\tilde{\xi}^{*}]\times [-1,1]^3 \subset \tilde{\Omega}_1$.

Now, we solve for $p$ in terms of $y$. By Equation~\eqref{eq:system1}, we can readily solve for $\hat{e} \cdot \hat{q}$:
\begin{equation}
\hat{e} \cdot \hat{q} = \frac{y \cdot \hat{q} (1 - a x \cdot \hat{e})}{a(1 - x \cdot y)}.
\end{equation}
Using this expression for $\hat{e} \cdot \hat{q}$ and $0 = y \cdot \hat{p} (\hat{e} \cdot \hat{p}) + y \cdot \hat{q} (\hat{e} \cdot \hat{q}) + x \cdot y (x \cdot \hat{e}) - y \cdot \hat{e}$, we get:
\begin{equation}
\hat{e} \cdot \hat{p} = \frac{y \cdot \hat{e} - x \cdot y (x \cdot \hat{e}) - \frac{(y \cdot \hat{q})^2}{a(1 - x \cdot y)} + \frac{(y \cdot \hat{q})^2 x \cdot \hat{e}}{1 - x \cdot y}}{y \cdot \hat{p}}.
\end{equation}

Likewise, we get:
\begin{equation}
\left\Vert p \right\Vert = \frac{\frac{(y \cdot \hat{q})^2}{y \cdot \hat{p}} - y \cdot \hat{p} + \frac{a(1 - x \cdot y)}{y \cdot \hat{p}(1 - a x \cdot \hat{e})}(y \cdot \hat{e} + x \cdot y (x \cdot \hat{e}))}{\alpha(1 - a x \cdot \hat{e}) - 1 + x \cdot y}.
\end{equation}

These are differentiable functions of $x \cdot y, y \cdot \hat{p}, y \cdot \hat{q}$ and $y \cdot \hat{e}$ provided that $y \cdot \hat{p} \neq 0$ and $(x \cdot y, x \cdot \hat{e}, y \cdot \hat{e}) \in \Omega$. If $\left\Vert p \right\Vert = 0$, then $y \cdot \hat{p} = 0$. Of course, in this case $\hat{e} \cdot \hat{p}$ and $\hat{e} \cdot \hat{q}$ are undefined, but, importantly, we know that $p=0$. Otherwise, using Equation~\eqref{eq:ypyx}, we see that $y \cdot \hat{p} = 0$ is only possible when:
\begin{equation}
x \cdot y = \frac{\left\Vert p \right\Vert(\alpha \beta + \gamma_1)+\gamma_2}{\left\Vert p \right\Vert(a \alpha \beta x \cdot \hat{e} - \gamma_1)-\gamma_2}.
\end{equation}
For small values of $\left\Vert p \right\Vert$, this shows that we can only have $y \cdot \hat{p} = 0$ when $x \cdot y$ is near $-1$. Knowing that the largest value of $x \cdot y$ that is allowable is $\tilde{\xi}^{*}$, we find the value of $\left\Vert p \right\Vert$ that yields this value of $x \cdot y$:
\begin{equation}
\left\Vert p \right\Vert = \frac{\gamma_2 (1 + \tilde{\xi}^{*})}{\tilde{\xi}^{*}(a \alpha \beta x \cdot \hat{e} - \gamma_1)-\gamma_2}.
\end{equation}
Denote this value, if it exists, as $p_3$. Then, define $\tilde{p}_1 = \min \left\{\tilde{p} ,p_3 \right\}$. We use this to define:
\begin{equation}
\Omega_2 = [0, \tilde{p}_1] \times [-1,1]^3,
\end{equation}
and then define
\begin{equation}
\Omega_1 = K(\Omega_2).
\end{equation}
Denote then $\xi^{*}_{1}$ as the largest value such that $[1,\xi^{*}_1]\times [-1,1]^3 \subset \Omega_1$.

We will identify $\Omega \subset [-1,1]^3$ as
\begin{equation}\label{eq:ptpinequality}
\Omega = \left\{(x \cdot y, x \cdot \hat{e}, y \cdot \hat{e}) \in [-1,1]^3 :  x \cdot y > 1+ 2\frac{a-1}{a+1} \geq 1- \frac{2}{1-a^2}(1 - a x \cdot \hat{e})(1 - a y \cdot \hat{e})  \right\}.
\end{equation}
We denote $\xi^{*}_2 = 1 + 2 \frac{a-1}{a+1}$. We will revisit the quantities $\tilde{p}_1$, $\xi^{*}_1$, and $\xi^{*}_2$ in Section~\ref{sec:regularity} to see how they are used to show the existence of solutions of the point-to-point problem.

Now, we move on to verifying more of the hypotheses on the cost function.

\begin{lemma}\label{lemma:Fbounds}
For the cost function arising from the point-to-point problem, $c(x,y) = F(x \cdot y, x \cdot \hat{e}, y \cdot \hat{e})$, we have $F_1, F_2, F_{11}, F_{12}, F_{13}, F_{22}, F_{23}$ are bounded functions on the set $\Omega$, defined in Equation~\eqref{eq:ptpinequality}.
\end{lemma}

\begin{proof}

On the set $\Omega$, we have:
\begin{equation}\label{eq:f1}
F_1(x \cdot y, x \cdot \hat{e}, y \cdot \hat{e}) = \frac{1}{\alpha(1-a x \cdot \hat{e})(1 - a y \cdot \hat{e})-1 + x \cdot y}<\infty
\end{equation}
where $\alpha = 2/(1-a^2)$. Also,
\begin{equation}\label{eq:f2}
F_2(x \cdot y, x \cdot \hat{e}, y \cdot \hat{e}) = \frac{-\frac{a(1-x \cdot y)}{1 - a x\cdot \hat{e}}}{\alpha(1-a x \cdot \hat{e})(1 - a y \cdot \hat{e}) - 1 + x \cdot y}<\infty,
\end{equation}

\begin{equation}\label{eq:f11}
F_{11}(x \cdot y, x \cdot \hat{e}, y \cdot \hat{e}) = \frac{-1}{(\alpha(1-a x \cdot \hat{e})(1 - a y \cdot \hat{e})-1 + x \cdot y)^2}<\infty,
\end{equation}

\begin{equation}\label{eq:f12}
F_{12}(x \cdot y, x \cdot \hat{e}, y \cdot \hat{e}) = \frac{a \alpha (1 - a y \cdot \hat{e})}{(\alpha(1-a x \cdot \hat{e})(1 - a y \cdot \hat{e})-1 + x \cdot y)^2}<\infty,
\end{equation}

\begin{equation}\label{eq:f13}
F_{13}(x \cdot y, x \cdot \hat{e}, y \cdot \hat{e}) = \frac{a \alpha (1 - a x \cdot \hat{e})}{(\alpha(1-a x \cdot \hat{e})(1 - a y \cdot \hat{e})-1 + x \cdot y)^2}<\infty,
\end{equation}

\begin{equation}\label{eq:f22}
F_{22}(x \cdot y, x \cdot \hat{e}, y \cdot \hat{e}) = F_2 \left( \frac{a}{1-ax \cdot \hat{e}} + \frac{a \alpha(1-a y \cdot \hat{e})}{\alpha(1-a x \cdot \hat{e})(1 - a y \cdot \hat{e}) - 1 + x \cdot y} \right)<\infty,
\end{equation}

\begin{equation}\label{eq:f23}
F_{23}(x \cdot y, x \cdot \hat{e}, y \cdot \hat{e}) = F_2 \left( \frac{a \alpha(1-a x \cdot \hat{e})}{\alpha(1-a x \cdot \hat{e})(1 - a y \cdot \hat{e}) - 1 + x \cdot y} \right)<\infty.
\end{equation}
\end{proof}

\begin{lemma}
The point-to-point cost function satisfies $F_2(1, x \cdot \hat{e}, y \cdot \hat{e}) = 0$ and $F_2(x \cdot y, x \cdot \hat{e}, y \cdot \hat{e}) \neq 0$ for $x \cdot y \neq 0$.
\end{lemma}

\begin{proof}
From Equation~\eqref{eq:f2}, we get that $\lim_{x \cdot y \rightarrow 1} F_2(x \cdot y, x \cdot \hat{e} ,y \cdot \hat{e}) = 0$ and is not equal to zero otherwise.
\end{proof}

\begin{lemma}\label{lemma:f2p}
For the cost function arising in the point-to-point problem, we have $F_2 / \left\Vert p \right\Vert$ is bounded as $\left\Vert p \right\Vert \rightarrow 0$.
\end{lemma}

\begin{proof}
We have that $\alpha_1= \mathcal{O}(1)$, $\alpha_2 = \mathcal{O}(1)$, $\beta_1 = \mathcal{O}(1)$ and $\beta_2 = \mathcal{O}(1)$ as $\left\Vert p \right\Vert \rightarrow 0$. Also, we compute:
\begin{align}
y \cdot \hat{p} &= (\gamma_2 + o(1))(1-(x \cdot y)) \\
y \cdot \hat{q} &= (\alpha_2 + o(1))(1-(x \cdot y)) \\
y \cdot \hat{e} &= x \cdot y \left( x \cdot \hat{e} - \gamma_2 (\hat{e} \cdot \hat{p}) - \alpha_2(\hat{e} \cdot \hat{q})+ o(1) \right) + \gamma_2 \left( \hat{e} \cdot \hat{p} \right) + \alpha_2 (\hat{e} \cdot \hat{q}) + o(1) \\
\alpha_1 &= \gamma_2 + \left\Vert p \right\Vert \left( \alpha \beta + \gamma_1 + a \alpha \beta \gamma_2 \hat{e} \cdot \hat{p} \right) + \mathcal{O}\left( \left\Vert p \right\Vert^2 \right) \\
\beta_1 &= -\gamma_2 - \left\Vert p \right\Vert \left( a \alpha \beta x \cdot \hat{e} - \gamma_1 - a \alpha \beta \gamma_2 \hat{e} \cdot \hat{p} \right) + \mathcal{O}\left( \left\Vert p \right\Vert^2 \right) \\
x \cdot y &= 1 + \mathcal{O} \left(\left\Vert p \right\Vert \right)
\end{align}
Therefore, since $F_2 = (1 - x \cdot y)\mathcal{O}(1)$ as $\left\Vert p \right\Vert \rightarrow 0$, we have $F_2 = \mathcal{O} \left( \left\Vert p \right\Vert \right)$ as $\left\Vert p \right\Vert \rightarrow 0$ and thus, $F_2 / \left\Vert p \right\Vert$ is bounded as $\left\Vert p \right\Vert \rightarrow 0$.
\end{proof}


\begin{lemma}
We show that, for $\xi^{*}<0$ the point-to-point cost function $\partial \left\Vert p \right\Vert / \partial (x \cdot y)<\infty$ for $x \cdot y = 0$.
\end{lemma}

\begin{proof}
We compute:
\begin{equation}\label{eq:det1}
\begin{vmatrix} 
1 & F_1 & 0 & F_{13} y \cdot \hat{p} + F_{23} y \cdot \hat{p} + F_{23} \hat{e} \cdot \hat{p} \\
0 & 0 & F_1 & 0 \\
0 & 2 y \cdot \hat{p} & 2 y \cdot \hat{q} & 0 \\
0 & \hat{e} \cdot \hat{p} & \hat{e} \cdot \hat{q} & -1
\end{vmatrix} = 2(y \cdot \hat{p}) F_1 \neq 0.
\end{equation}
Since $F_1 \neq 0$, this reduces to showing that $y \cdot \hat{p} \neq 0$ when $x \cdot y = 0$. This would require that:

\begin{equation}
\pm 1 = \frac{a \hat{e} \cdot \hat{q}}{1 - a (x \cdot \hat{e})}.
\end{equation}
We show that this is impossible. Denote $\tilde{e}$ to be the orthogonal projection of $\hat{e}$ onto the plane defined by $x$ and $\hat{q}$ and $\rho = \left\Vert \tilde{e} \right\Vert \leq 1$. Thus, we examine if it is possible for:
\begin{equation}
1 = \frac{ \rho a \sqrt{1-(x \cdot \tilde{e})^2}}{1 - \rho a (x \cdot \tilde{e})}.
\end{equation}

This becomes:
\begin{equation}
2 \rho^2 a^2(x \cdot \tilde{e})^2 - 2 \rho a (x \cdot \tilde{e}) + 1-\rho^2 a^2 = 0.
\end{equation}

This does not have real roots if $a <1/\sqrt{2} \leq 1/\sqrt{2} \rho$. This is fine, since, via Inequality~\eqref{eq:ptpinequality}, we get:
\begin{equation}
1 - \frac{2}{1-a^2}(1 - a x \cdot \hat{e})(1 - a y \cdot \hat{e}) \leq 1 - \frac{2(1-a)^2}{1-a^2},
\end{equation}
and this only is greater than or equal to zero for $a > 1/3$. Thus, we are guaranteed to not have real roots and thus the determinant in Equation~\eqref{eq:det1} is nonzero. This then proves the boundedness of $\partial \left\Vert p \right\Vert / \partial (x \cdot y)$ for the point-to-point cost function by finally using the result of Lemma~\ref{lemma:Fbounds}.
\end{proof}

\subsection{Mixed Hessian}\label{sec:mixedHessian}

In this section, we show that the mixed Hessian term is non-zero on $D_{\gamma}$, defined in Equation~\eqref{eq:dgamma}. We derive a formula for the mixed Hessian, whose geometric setup is shown in Figure~\ref{fig:formula2}. The derivation stems from the fact that
\begin{equation}
\left\vert D_{p} T \right\vert = \frac{1}{\left\vert \det D_{xy}c(x,y) \right\vert},
\end{equation}
where $y = T(x,p)$.

\begin{figure}[htp]
	\centering
	\includegraphics[width=0.8\textwidth]{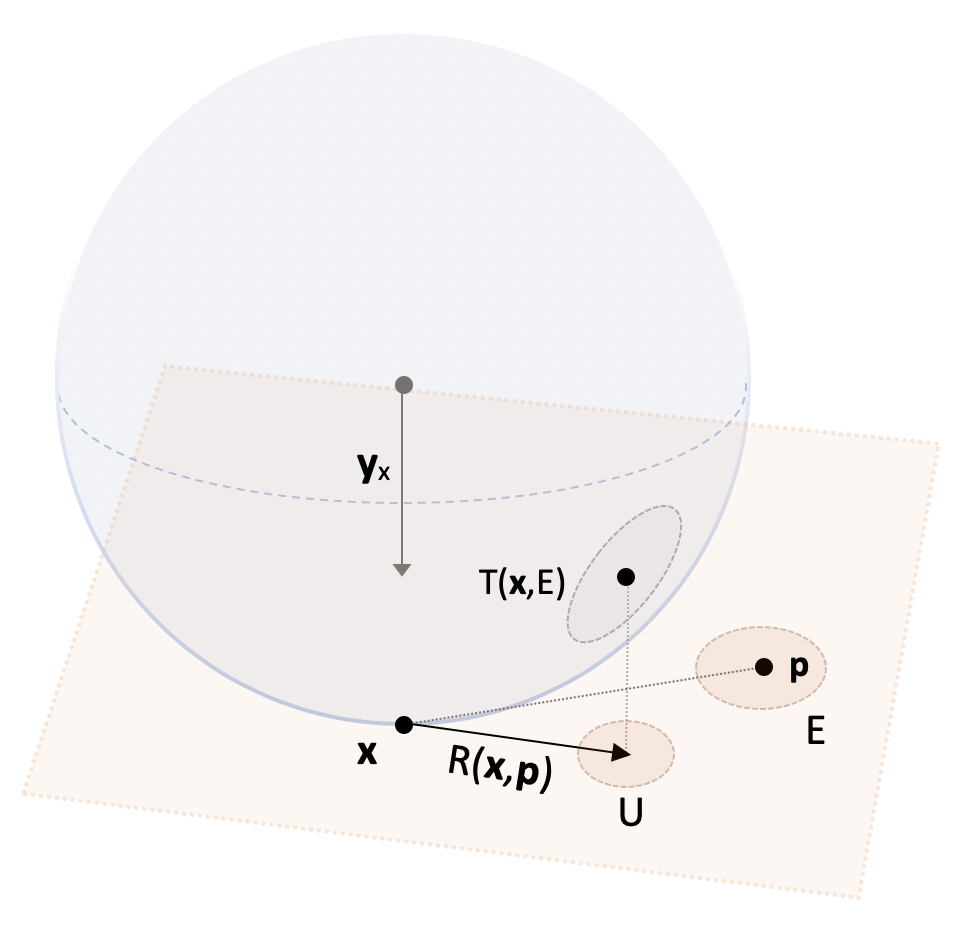}\label{fig:formula2}
	\caption{Change in area formula from tangent coordinates $p$ to coordinates on the sphere $T(x,p) = (x \cdot y) x + (y \cdot \hat{p}) p + (y \cdot \hat{q}) q$ via the coordinates $u$ of the orthogonal projection of $T(x,p)$ onto the tangent plane $\mathcal{T}_{x}$ at $x$.}
	\label{fig:formula2}
\end{figure}

Let $T(x,p) = (x \cdot y) x + (y \cdot \hat{p}) p + (y \cdot \hat{q}) q =  (x \cdot y) x + R(p)$. Then, defining $U = R(E) = \left\{ (y \cdot \hat{p}) \hat{p} + (y \cdot \hat{q}) \hat{q} \vert p \in E \right\}$. Then, $T(x,E) = \left\{ \left(u_1, u_2, \sqrt{1-u_1^2-u_2^2} \right) \vert (u_1, u_2) \in U \right\}$. We have:
\begin{equation}
\int_{E} \left\vert \det D_{p} T \right\vert dp = \int_{T(x,E)} dS
\end{equation}

\begin{equation}
= \int_{U} \frac{1}{\sqrt{1-\left\Vert u \right\Vert^2}} du
\end{equation}

\begin{equation}
= \int_{E} \frac{1}{\sqrt{1- \left\vert R(p) \right\vert^2}} \left\vert \det \nabla R(p) \right\vert dp
\end{equation}
and therefore,
\begin{equation}
\left\vert \det D_{p} T \right\vert = \frac{\left\vert \det \nabla R(p) \right\vert}{\sqrt{1-\left\vert R(p) \right\vert^2}}
\end{equation}
and
\begin{equation}
\left\vert \det D^2_{xy} c \right\vert = \frac{\sqrt{1-\left\vert R(p) \right\vert^2}}{\left\vert \det \nabla R(p) \right\vert}
\end{equation}

We also define:
\begin{equation}
T(x, p) = R_1 x + R_2 p + R_3 q.
\end{equation}

Thus, we get:
\begin{equation}\label{eq:mixedHessian}
\left\vert \det D^2_{xy} c \right\vert = \frac{\left\vert R_1 \right\vert}{\left\vert \det \nabla R(p) \right\vert}.
\end{equation}

We can derive an explicit expression for the mixed Hessian in terms of $R_1, R_2$ and $R_3$. Since $R(p) = R_2 p + R_3 q = R_2(p) (p_1, p_2) + R_3(p) (-p_2, p_1) = (p_1R_2 -p_2R_3, p_2R_2 + p_1R_3)$. Therefore,
\begin{equation}
\nabla R(p) = \begin{pmatrix}
R_2 + p_1 \frac{\partial R_2}{\partial p_1} - p_2 \frac{\partial R_3}{\partial p_1} & p_1 \frac{\partial R_2}{\partial p_2} - R_3 - p_2 \frac{\partial R_3}{\partial p_2} \\
p_2 \frac{\partial R_2}{\partial p_1} + R_3 + p_1 \frac{\partial R_3}{\partial p_1} & R_2 + p_2\frac{\partial R_2}{\partial p_2} + p_1 \frac{\partial R_3}{\partial p_2}.
\end{pmatrix}
\end{equation}

Thus,
\begin{multline}
\det \nabla R(p) = 
\left( R_2 + p_1 \frac{\partial R_2}{\partial p_1} - p_2 \frac{\partial R_3}{\partial p_1}\right)\left( R_2 + p_2\frac{\partial R_2}{\partial p_2} + p_1 \frac{\partial R_3}{\partial p_2} \right) - \\
\left(p_1 \frac{\partial R_2}{\partial p_2} - R_3 - p_2 \frac{\partial R_3}{\partial p_2} \right) \left( p_2 \frac{\partial R_2}{\partial p_1} + R_3 + p_1 \frac{\partial R_3}{\partial p_1} \right)
\end{multline}

\begin{multline}
= R_2^2 + p_2R_2 \frac{\partial R_2}{\partial p_2}  + p_1 R_2 \frac{\partial R_3}{\partial p_2}  + p_1R_2 \frac{\partial R_2}{\partial p_1} + p_1p_2\frac{\partial R_2}{\partial p_1} \frac{\partial R_2}{\partial p_2} +p_1^2 \frac{\partial R_2}{\partial p_1} \frac{\partial R_3}{\partial p_2}  - \\
p_2 R_2 \frac{\partial R_3}{\partial p_1}-p_2^2 \frac{\partial R_3}{\partial p_1} \frac{\partial R_2}{\partial p_2} - p_1p_2 \frac{\partial R_3}{\partial p_1} \frac{\partial R_3}{\partial p_2}  -p_1p_2 \frac{\partial R_2}{\partial p_2}\frac{\partial R_2}{\partial p_1} -p_1R_3 \frac{\partial R_2}{\partial p_2} - p_1^2 \frac{\partial R_2}{\partial p_2} \frac{\partial R_3}{\partial p_1} + \\
p_2 R_3 \frac{\partial R_2}{\partial p_1}  + R_3^2 + p_1 R_3 \frac{\partial R_3}{\partial p_1} + p_2^2 \frac{\partial R_3}{\partial p_2} \frac{\partial R_2}{\partial p_1}  + p_2 R_3 \frac{\partial R_3}{\partial p_2} + p_1 p_2 \frac{\partial R_3}{\partial p_2} \frac{\partial R_3}{\partial p_1} 
\end{multline}

\begin{multline}
= R_2^2 + R_3^2 + p_1R_2 \frac{\partial R_2}{\partial p_1}+ p_2R_2 \frac{\partial R_2}{\partial p_2}  + p_1 R_3 \frac{\partial R_3}{\partial p_1}+ p_2 R_3 \frac{\partial R_3}{\partial p_2}+ \\
+p_1 R_2 \frac{\partial R_3}{\partial p_2} - p_2 R_2 \frac{\partial R_3}{\partial p_1} + p_2 R_3 \frac{\partial R_2}{\partial p_1} -p_1R_3 \frac{\partial R_2}{\partial p_2} + \\
+p_1^2 \frac{\partial R_2}{\partial p_1} \frac{\partial R_3}{\partial p_2}  - p_1^2 \frac{\partial R_2}{\partial p_2} \frac{\partial R_3}{\partial p_1} -p_2^2 \frac{\partial R_3}{\partial p_1} \frac{\partial R_2}{\partial p_2} + p_2^2 \frac{\partial R_3}{\partial p_2} \frac{\partial R_2}{\partial p_1} 
\end{multline}

Let $(\nabla R_i)^{\perp} = \left(\frac{\partial R_i}{\partial p_2}, -\frac{\partial R_i}{\partial p_1} \right)$ for $i=2,3$. Then, we finally get
\begin{multline}\label{eq:fullmH}
\det \nabla R(p) = R_2^2 + R_3^2 + R_2 \nabla R_2 \cdot p + R_3 \nabla R_3 \cdot p + \left\Vert p \right\Vert^2 \nabla R_2 \cdot (\nabla R_3)^{\perp} - \\
R_2 (\nabla R_3)^{\perp} \cdot p + R_3 (\nabla R_2)^{\perp} \cdot p.
\end{multline}

Our goal in this section is to prove that for cost functions satisfying Hypothesis~\ref{hyp:mainhypothesis}, we have $\left\vert D^{2}_{xy} c(x,y) \right\vert \neq 0$ for some domain $D \subset \mathbb{S}^2 \times \mathbb{S}^2$. We begin by showing that the denominator of Equation~\eqref{eq:mixedHessian} never goes to infinity.

\begin{lemma}\label{lemma:boundedness}
Provided that $F_1, F_2, F_{11}, F_{12}, F_{13}, F_{22}, F_{23}$ are bounded, $F_1 \neq 0$ and $F_2\left\Vert p \right\Vert$ is bounded as $\left\Vert p \right\Vert \rightarrow 0$, then $\left\vert \det \nabla R(p) \right\vert$ is bounded.
\end{lemma}

\begin{proof}
Since:
\begin{align}
R_2 p &= (y \cdot \hat{p})\hat{p}, \\
R_3 q &= (y \cdot \hat{q})\hat{q},
\end{align}
we have:
\begin{align}
\hat{p} R_2 + \left\Vert p \right\Vert \nabla R_2 &= \nabla(y \cdot \hat{p}), \\
\hat{p} R_3 + \left\Vert p \right\Vert \nabla R_3 &= \nabla(y \cdot \hat{q}).
\end{align}

Thus, we find that we can control $\nabla R_2$ by $R_2$ and $\nabla(y \cdot \hat{p})$ and $\nabla R_3$ by $R_3$ and $\nabla(y \cdot \hat{q})$. Since, by the assumptions on our cost function, we have:
\begin{align}
y \cdot \hat{p} &= K^2(\left\Vert p \right\Vert, x \cdot \hat{e}, \hat{e} \cdot \hat{p}, \hat{e} \cdot \hat{q}), \\
y \cdot \hat{q} &= K^3(\left\Vert p \right\Vert, x \cdot \hat{e}, \hat{e} \cdot \hat{p}, \hat{e} \cdot \hat{q}),
\end{align}
where $K^{i}$ denotes the $i$th entry of the vector-valued function $K$ that satisfies Equation~\eqref{eq:HK}. We thus compute:
\begin{align}
\nabla (y \cdot \hat{p})= \nabla K^2 &= \hat{p} K^2_{1} + \gamma_{1} K^2_{3} + \gamma_{2}K^{2}_{4}, \\
\nabla (y \cdot \hat{q}) = \nabla K^3 &= \hat{p} K^3_{1} + \gamma_{1} K^3_{3} + \gamma_{2}K^{3}_{4},
\end{align}
where $\gamma_1$ and $\gamma_2$ are vectors of length at most $1$. We know that all $K^{i}_{j}$ are bounded by the assumptions on the cost function. Now, we just need bounds on $R_2$ and $R_3$. From Equation~\eqref{eq:system1}, we get:
\begin{align}
R_2(\left\Vert p \right\Vert) &= \frac{y \cdot \hat{p}}{\left\Vert p \right\Vert} = -\frac{1}{F_1(x \cdot y, x \cdot \hat{e}, y \cdot \hat{e})}-\frac{F_2(x \cdot y, x \cdot \hat{e}, y \cdot \hat{e})}{\left\Vert p \right\Vert F_1(x \cdot y, x \cdot \hat{e}, y \cdot \hat{e})}\hat{e} \cdot \hat{p} \\
R_3(\left\Vert p \right\Vert) &= \frac{y \cdot \hat{q}}{\left\Vert p \right\Vert} = -\frac{F_2(x \cdot y, x \cdot \hat{e}, y \cdot \hat{e})}{\left\Vert p \right\Vert F_1(x \cdot y, x \cdot \hat{e}, y \cdot \hat{e})}\hat{e} \cdot \hat{q}.
\end{align}

Thus, by the Hypotheses~\ref{hyp:mainhypothesis}, $F_1$ is bounded on $\Omega$ and $F_2/\left\Vert p \right\Vert$ is bounded as $\left\Vert p \right\Vert \rightarrow 0$. The quantity $\left\vert \det \nabla R(p) \right\vert$ can be bounded due to bounds on $R_2, R_3, \nabla R_2, \nabla R_3$ using the explicit expression derived in Equation~\eqref{eq:fullmH}.
\end{proof}



Now that we have shown that the denominator of Equation~\eqref{eq:mixedHessian} never blows up. We need to verify that as $R_1 = x \cdot y \rightarrow 0$ (if this is allowed), the denominator $\left\vert \det \nabla R(p) \right\vert = o(x \cdot y)$, so therefore $\lim_{x \cdot y \rightarrow 0} \left\vert D^2_{xy} c(x,y) \right\vert \neq 0$.

\begin{lemma}\label{lemma:notzero}
Assuming the cost function satisfies Hypothesis~\ref{hyp:mainhypothesis}, we have

\begin{equation}
\lim_{x \cdot y \rightarrow 0} \left\vert D^2_{xy} c(x,y) \right\vert \neq 0.
\end{equation}
\end{lemma}

\begin{proof}
If $\xi^{*}>0$, then we are done.

If $\xi^{*}\leq 0$, we need to show that as $R_1 = x \cdot y \rightarrow 0$, the mixed Hessian~\eqref{eq:mixedHessian} does not approach zero. Fix $x, \hat{e}, \hat{p}$. Let $\hat{s}$ and $\hat{z} = x \times \hat{s}$ be fixed directions in the tangent plane $T_{x}$, chosen such that:
\begin{equation}
R(p) = (y \cdot \hat{s})\hat{s},
\end{equation}
precisely when $y$ satisfies $x \cdot y = 0$.  To be clear, $R(p)$ is not necessarily equal to $(y \cdot \hat{s}) \hat{s}$ in general. Then, the function $y \cdot \hat{s}$ has a local maximum as a function of $\left\Vert p \right\Vert$, since $y \cdot \hat{s}$ attains the maximum value of $1$ only when $x \cdot y = 0$. Likewise, since $y \cdot \hat{z} = 0$ when $x \cdot y = 0$, the function $y \cdot \hat{z}$ has a local minimum when $x \cdot y = 0$.

Let us first see how $\det \nabla R(p) \rightarrow 0$ as $x \cdot y \rightarrow 0$. Fix $x$, $\hat{e}$ and a coordinate system $(u_1, u_2)$ in the tangent plane. We compute:
\begin{equation}\label{eq:Rpmatrix}
\nabla R(p) = \begin{pmatrix} \frac{\partial}{\partial p_1} (y \cdot \hat{s})s_1 + \frac{\partial}{\partial p_1} (y \cdot \hat{z})z_1 & \frac{\partial}{\partial p_2} (y \cdot \hat{s})s_1+\frac{\partial}{\partial p_2} (y \cdot \hat{z})z_1 \\ \frac{\partial}{\partial p_1} (y \cdot \hat{s})s_2+\frac{\partial}{\partial p_1} (y \cdot \hat{z})z_2 & \frac{\partial}{\partial p_2} (y \cdot \hat{s})s_2 + \frac{\partial}{\partial p_2} (y \cdot \hat{z})z_2 \end{pmatrix}.
\end{equation}
After an orthogonal change of coordinates, defined by the tangent coordinates $\hat{t}_1 = \hat{p}, \hat{t}_2 = \hat{q}$, we have new coordinates $(v_1, v_2)$. We now fix this choice of coordinates. Thus, the new derivative $\partial (y \cdot \hat{s})/\partial p_1$ computes the change in $(y \cdot \hat{s})$ given perturbations in the direction of $\hat{p}$. This is precisely changes in $\left\Vert p \right\Vert$, so all the partial derivatives in the first column become equal to zero when $x \cdot y = 0$. Therefore, $\det \nabla R(p) = 0$ when $x \cdot y = 0$.

Now we will show that, asymptotically, $\lim_{x \cdot y \rightarrow 0} \det \nabla R(p) = o(x \cdot y)$. Let $x \cdot y = \xi \left( \left\Vert p \right\Vert, x \cdot \hat{e}, \hat{e} \cdot \hat{p}, \hat{e} \cdot \hat{q} \right)$. Fix $x, \hat{e}, \hat{p}$ and let $p^{*}$ denote the smallest value of $\left\Vert p \right\Vert$ such that
\begin{equation}\label{eq:zerocondition}
0 = \xi \left( p^{*}, x \cdot \hat{e}, \hat{e} \cdot \hat{p}, \hat{e} \cdot \hat{q} \right).
\end{equation}

The functions $y \cdot \hat{s}$ and $y \cdot \hat{z}$ are local maxima and minima at $p^{*}$, respectively, as a function of $\left\Vert p \right\Vert$, and $p^{*}-\left\Vert p \right\Vert = L(x \cdot y) + \mathcal{O} \left( (x \cdot y)^2 \right)$ by Hypothesis~\ref{hyp:mainhypothesis} and by Equation~\eqref{eq:zerocondition} for some constant $0<L<\infty$ as $x \cdot y \rightarrow 0$. Therefore, we may write
\begin{equation}
y \cdot \hat{s} = 1 - c\left(p^{*} - \left\Vert p \right\Vert \right)^2 + \mathcal{O}\left(p^{*} - \left\Vert p \right\Vert \right)^3
\end{equation}
as $\left\Vert p \right\Vert \rightarrow p^{*}$ for some $c\geq0$. Thus,
\begin{equation}
\frac{\partial}{\partial p_1}(y \cdot \hat{s}) =  2c\left(p^{*} - \left\Vert p \right\Vert \right) + \mathcal{O}\left(p^{*} - \left\Vert p \right\Vert \right)^2
\end{equation}
as $\left\Vert p \right\Vert \rightarrow p^{*}$. Thus, $y \cdot \hat{s} = \mathcal{O}(x \cdot y)$ as $x \cdot y \rightarrow 0$ and, likewise, $\frac{\partial}{\partial p_1}(y \cdot \hat{z}) = \mathcal{O}(x \cdot y)$ as $x \cdot y \rightarrow 0$. Hence, by Equation~\eqref{eq:Rpmatrix}, we have $\left\vert \det \nabla R(p) \right\vert = \mathcal{O}(x \cdot y)$ as $x \cdot y \rightarrow 0$, which means that there exist constants $\zeta_0$ and $C>0$ such that $\left\vert \det \nabla R(p) \right\vert \leq C \left\vert x \cdot y \right\vert $ for $\left\vert x \cdot y \right\vert \leq \zeta_0$ and thus,
\begin{equation}
\lim_{x \cdot y \rightarrow 0} \left\vert D^2_{xy} c \right\vert = \lim_{x \cdot y \rightarrow 0} \frac{\left\vert x \cdot y \right\vert}{\left\vert \det \nabla R(p) \right\vert} \geq \lim_{x \cdot y \rightarrow 0} \frac{\left\vert x \cdot y \right\vert}{C\left\vert x \cdot y \right\vert} = C>0.
\end{equation}
\end{proof}

Then, we have the following result.

\begin{theorem}
For cost functions satisfying Hypothesis~\ref{hyp:mainhypothesis}, $\left\vert D^{2}_{xy} c(x,y) \right\vert \neq 0$ on $D_{\gamma} \subset \mathbb{S}^2 \times \mathbb{S}^2$, where $D_{\gamma}$ is defined by Equation~\eqref{eq:dgamma}.
\end{theorem}

\begin{proof}
The result follows, given that the cost function satisfies Hypothesis~\ref{hyp:mainhypothesis}, then Lemma~\ref{lemma:boundedness} and Lemma~\ref{lemma:notzero} both hold. Therefore, by the definition of $D_{\gamma}$, we can claim that $\left\vert D^{2}_{xy} c(x,y) \right\vert \neq 0$ on $D_{\gamma}$.
\end{proof}

\begin{corollary}
The point-to-point cost function satisfies $\left\vert D^{2}_{xy} c(x,y) \right\vert \neq 0$ on $D_{\gamma} \subset \mathbb{S}^2 \times \mathbb{S}^2$.
\end{corollary}

\begin{proof}
This follows from the computations in Section~\ref{sec:verify}.
\end{proof}

\subsection{Cost-Sectional Curvature}

We compute the $3 \times 3$ matrix Hessian for the cost function $c(x,y) = F(x \cdot y, x \cdot \hat{e}, y \cdot \hat{e})$ using Euclidean derivatives, since we use the metric on the sphere induced by the standard Euclidean metric. Given a function $f: \mathbb{R}^3 \rightarrow \mathbb{R}$, we compute the Hessian on the sphere as follows:
\begin{equation}
\nabla^{2}_{xx} f(x) = D^2_{xx} f(x) - (D_{x} f(x) \cdot x) \text{Id},
\end{equation}
where $D$ are the standard (three-dimensional) Euclidean (ambient) derivatives. We will thus compute the term $\nabla^{2}_{xx} c(x,y) \vert_{y = T}$. We get:
\begin{equation}
D_{x} F(x \cdot y, x \cdot \hat{e}, y \cdot \hat{e}) \cdot x = (x \cdot y)F_1 + (x \cdot \hat{e})F_2
\end{equation}

We also compute:
\begin{equation}
\frac{\partial^2}{\partial x_{i} x_{j}} F(x \cdot y, x \cdot \hat{e}, y \cdot \hat{e}) = y_{i} y_{j} F_{11} + y_{i}e_{j}F_{12} + y_{j}e_{i}F_{12} + e_{i}e_{j}F_{22}
\end{equation}

Thus, we get:
\begin{multline}
\nabla^{2}_{xx} F(x \cdot y, x \cdot \hat{e}, y \cdot \hat{e}) = \\
F_{11} A + F_{12} (B + B^{T}) +F_{22} C - \left( (x \cdot y)F_1 + (x \cdot \hat{e})F_2 \right) \begin{pmatrix} 1 & 0 & 0 \\ 0 & 1 & 0 \\ 0 & 0 & 1 \end{pmatrix},
\end{multline}
where $(A)_{ij} = y_{i}y_{j}$, $(B)_{ij} = e_{i}y_{j}$, and $(C)_{ij} = e_{i}e_{j}$.


Recall that
\begin{equation}
y_{i} = x_i(x \cdot y) + \frac{p_i}{\left\Vert p \right\Vert}(y \cdot \hat{p}) + \frac{q_i}{\left\Vert p \right\Vert} (y \cdot \hat{q}),
\end{equation}



and
\begin{equation}
q_{i} = x_{\phi(i+1)}p_{\phi(i+2)} - x_{\phi(i+2)}p_{\phi(i+1)},
\end{equation}
where $\phi(i) = \left( (i-1) \ \text{mod} \ 3 \right)+1$. This Hessian can thus be entirely expressed as a function of $x, p$. Thus, we have:
\begin{multline}
\left( \nabla^{2}_{xx} c(x, y) \vert_{y = T} \right)_{kl} = f_1(p)x_kx_l +f_2(p) \left( x_kp_l + x_l p_k \right) +f_3(p) \left( x_kq_l + x_l q_k \right) + \\
f_4(p)p_kp_l +f_5(p) \left( p_kq_l + p_l q_k \right) + f_6(p)q_kq_l + f_7(p)(e_{k}x_{l}+e_{l}x_{k}) +f_8(p)(e_k p_l + e_l p_k)+ \\
f_9(p)(e_k q_l + e_l q_k) + f_{10}(p)e_{k}e_{l} + f_{11}(p)\delta_{kl},
\end{multline}
where
\begin{align}\label{eq:ffunctions}
f_1(p) &= (x \cdot y)^2F_{11} \\
f_2(p) &= (x \cdot y)(y \cdot \hat{p})F_{11}/\left\Vert p \right\Vert =  \\
f_3(p) &= (x \cdot y)(y \cdot \hat{q})F_{11}/\left\Vert p \right\Vert \\
f_4(p) &= (y \cdot \hat{p})^2F_{11}/\left\Vert p \right\Vert^2 \\
f_5(p) &= (y \cdot \hat{p})(y \cdot \hat{q})F_{11}/\left\Vert p \right\Vert^2 \\
f_6(p) &= (y \cdot \hat{q})^2F_{11}/\left\Vert p \right\Vert^2 \\
f_7(p) &= (x \cdot y)F_{12} \\
f_8(p) &= (y \cdot \hat{p})F_{12}/\left\Vert p \right\Vert \\
f_9(p) &= (y \cdot \hat{q})F_{12}/\left\Vert p \right\Vert \\
f_{10}(p) &= F_{22} \\
f_{11}(p) &= -(x \cdot y)F_1 - (x \cdot \hat{e})F_2
\end{align}

Now, we proceed to take derivatives with respect to $p$:
\begin{multline}
D_{p_{i}}\left( \nabla^{2}_{xx} c(x, y) \vert_{y = T} \right)_{kl} = \\
D_{p_{i}}f_1(p)x_kx_l +D_{p_{i}}f_2(p) (x_{k}p_l + x_l p_k) + f_2(p) (x_k \delta_{il} + x_l \delta_{ik}) + D_{p_{i}}f_3(p) (x_k q_l + x_l q_k) + \\
f_3(p) \left( x_k \left(x_{\phi(l+1)} \delta_{i \phi(l+2)} - x_{\phi(l+2)}\delta_{i \phi(l+1)} \right)+x_l \left(x_{\phi(k+1)} \delta_{i \phi(k+2)} - x_{\phi(k+2)}\delta_{i \phi(k+1)} \right) \right) + \\
D_{p_{i}} f_4(p) p_k p_l + f_{4}(p) (\delta_{ik} p_l + p_{k} \delta_{il}) + D_{p_{i}} f_5(p) (p_k q_l+p_l q_k) + \\
f_5(p) \big( \delta_{ik}q_{l} + p_k(x_{\phi(l+1)} \delta_{i \phi(l+2)} - x_{\phi(l+2)} \delta_{i\phi(l+1)})+ \\
\delta_{il}q_{k} + p_l(x_{\phi(k+1)} \delta_{i \phi(k+2)} - x_{\phi(k+2)} \delta_{i\phi(k+1)}) \big) + \\
D_{p_{i}}f_6(p) q_k q_l + f_6(p) \left( x_{\phi(k+1)}\delta_{i \phi(k+2)} - x_{\phi(k+2)} \delta_{i\phi(k+1)} \right)q_l + \\
f_6(p) \left(x_{\phi(l+1)} \delta_{i\phi(l+2)} - x_{\phi(l+2)} \delta_{i \phi(l+1)} \right)q_k +D_{p_{i}} f_{7}(p) (e_k x_l + e_l x_k) + \\
D_{p_{i}}f_8(p)(e_k p_l + e_l p_k) +f_8(p)(e_k \delta_{il} + e_{l} \delta_{ik}) + D_{p_{i}}f_9(p) (e_k q_l + e_l q_k) + \\
f_9(p) \left( e_k(x_{\phi(l+1)}\delta_{i\phi(l+2)} - x_{\phi(l+2)}\delta_{i\phi(l+1)}) + e_l(x_{\phi(k+1)}\delta_{i\phi(k+2)} - x_{\phi(k+2)} \delta_{i\phi(k+1)}) \right) + \\
D_{p_{i}}f_{10}(p) e_{k}e_{l} + D_{p_{i}}f_{11}(p) \delta_{kl}
\end{multline}

Taking another derivative, we get
\begin{multline}
D_{p_{i}p_{j}}\left( \nabla^{2}_{xx} c(x, y) \vert_{y = T} \right)_{kl} = \\
D_{p_{i}p_{j}}f_1(p) x_k x_l + D_{p_{i} p_{j}}f_{2}(p) (x_k p_l+x_l p_k) + D_{p_{i}}f_2(p) (x_k \delta_{jl}+x_l \delta_{jk}) + \\
D_{p_{j}} f_2(p) (x_k \delta_{il}+ x_l \delta_{ik}) + 
D_{p_{i}p_{j}} f_3(p) (x_k q_l+x_l q_k) + \\
D_{p_{i}}f_3(p) \big( x_k \left( x_{\phi(l+1)} \delta_{j\phi(l+2)} - x_{\phi(l+2)}\delta_{j\phi(l+1)} \right)+ \\
x_l \left( x_{\phi(k+1)} \delta_{j\phi(k+2)} - x_{\phi(k+2)}\delta_{j\phi(k+1)} \right) \big)+ \\
D_{p_{j}}f_3(p)\big(x_k \left(x_{\phi(l+1)} \delta_{i \phi(l+2)} - x_{\phi(l+2)}\delta_{i \phi(l+1)} \right)+\\
x_l \left(x_{\phi(k+1)} \delta_{i \phi(k+2)} - x_{\phi(k+2)}\delta_{i \phi(k+1)} \right) \big) + \\
D_{p_{i}p_{j}}f_4(p)p_k p_l + D_{p_{i}} f_4(p) \left( p_k \delta_{jl} + p_{l} \delta_{jk} \right) + D_{p_{j}}f_4(p) \left( \delta_{ik}p_l + \delta_{il}p_k \right) + \\
f_4(p) \left( \delta_{ik}\delta_{jl} + \delta_{jk}\delta_{il} \right) + D_{p_{i}p_{j}}f_5(p) (p_k q_l+p_l q_k) + \\
D_{p_{i}}f_5(p) \big( \delta_{jk}q_l + p_k(x_{\phi(l+1)}\delta_{j\phi(l+2)} - x_{\phi(l+2)} \delta_{j\phi(l+1)}) + \\
\delta_{jl}q_k + p_l(x_{\phi(k+1)}\delta_{j\phi(k+2)} - x_{\phi(k+2)} \delta_{j\phi(k+1)}) \big) + \\
D_{p_{j}}f_5(p) \big( \delta_{ik}q_{l} + p_k(x_{\phi(l+1)} \delta_{i \phi(l+2)} - x_{\phi(l+2)} \delta_{i\phi(l+1)})+ \\
\delta_{il}q_{k} + p_l(x_{\phi(k+1)} \delta_{i \phi(k+2)} - x_{\phi(k+2)} \delta_{i\phi(k+1)}) \big) + \\
f_5(p) \left( \delta_{ik}(x_{\phi(l+1)}\delta_{j \phi(l+2)} - x_{\phi(l+2)}\delta_{j\phi(l+1)}) + \delta_{jk} \left( x_{\phi(l+1)}\delta_{i \phi(l+2)} - x_{\phi(l+2)} \delta_{i \phi(l+1)} \right) \right) + \\
f_5(p) \left( \delta_{il}(x_{\phi(k+1)}\delta_{j \phi(k+2)} - x_{\phi(k+2)}\delta_{j\phi(k+1)}) + \delta_{jl} \left( x_{\phi(k+1)}\delta_{i \phi(k+2)} - x_{\phi(k+2)} \delta_{i \phi(k+1)} \right) \right) + \\
D_{p_{i}p_{j}} f_6(p) q_k q_l + \\
D_{p_{i}}f_6(p) \left(q_l(x_{\phi(k+1)}\delta_{j \phi(k+2)} - x_{\phi(k+2)}\delta_{j\phi(k+1)}) + q_k(x_{\phi(l+1)} \delta_{j\phi(l+2)} - x_{\phi(l+2)} \delta_{j\phi(l+2)}) \right) + \\
D_{p_{j}}f_6(p) \left( x_{\phi(k+1)}\delta_{i \phi(k+2)} - x_{\phi(k+2)} \delta_{i\phi(k+1)} \right)q_l + \\
f_6(p) \left( x_{\phi(k+1)}\delta_{i \phi(k+2)} - x_{\phi(k+2)} \delta_{i\phi(k+1)} \right) \left(x_{\phi(l+1)} \delta_{j\phi(l+2)}-x_{\phi(l+2)}\delta_{j\phi(l+1)} \right) + \\
D_{p_{j}}f_6(p) \left(x_{\phi(l+1)} \delta_{i\phi(l+2)} - x_{\phi(l+2)} \delta_{i \phi(l+1)} \right)q_k + \\
f_6(p) \left(x_{\phi(l+1)} \delta_{i\phi(l+2)} - x_{\phi(l+2)} \delta_{i \phi(l+1)} \right) \left( x_{\phi(k+1)}\delta_{j \phi(k+2)} - x_{\phi(k+2)} \delta_{j\phi(k+2)}\right) + \\
D_{p_{i}p_{j}}f_7(p) \left( e_kx_l+e_lx_k \right) + D_{p_{i}p_{j}}f_8(p) (e_kp_l+e_lp_k) + D_{p_{i}}f_8(p)(e_k \delta_{jl} + e_l \delta_{jk}) + \\
D_{p_{j}}f_8(p)(e_k\delta_{il}+e_l\delta_{ik}) + D_{p_{i}p_{j}}f_9(p) (e_kq_l+e_lq_k) + \\
D_{p_{i}}f_9(p) \left(e_k(x_{\phi(l+1)}\delta_{j\phi(l+2)}-x_{\phi(l+2)}\delta_{j\phi(l+1)}) + e_l(x_{\phi(k+1)}\delta_{j\phi(k+2)} - x_{\phi(k+2)} \delta_{j \phi(k+1)}) \right) + \\
D_{p_{j}}f_9(p) \left( e_k(x_{\phi(l+1)}\delta_{i\phi(l+2)} - x_{\phi(l+2)}\delta_{i\phi(l+1)}) + e_l(x_{\phi(k+1)}\delta_{i\phi(k+2)} - x_{\phi(k+2)} \delta_{i\phi(k+1)}) \right) + \\
D_{p_{i}p_{j}}f_{10}(p)e_ke_l + D_{p_{i}p_{j}}f_{11}(p)\delta_{kl}
\end{multline}

Now, we compute:
\begin{multline}
\sum_{i, j, k, l}D_{p_{i}p_{j}}\left( \nabla^{2}_{xx} c(x, y) \vert_{y = T} \right)_{kl} \xi_{i}\xi_j\eta_k \eta_l = \\
\left\langle D^2 f_1(p) \xi, \xi \right\rangle (x \cdot \eta)^2 + 2\left\langle D^2 f_2(p) \xi, \xi \right\rangle (x \cdot \eta)(p \cdot \eta) + \\
4(Df_2(p) \cdot \xi)(x \cdot \eta)(\xi \cdot \eta) + 2\left\langle D^2 f_3(p) \xi, \xi \right\rangle (x \cdot \eta)(q \cdot \eta) + 4(Df_3(p)\cdot \xi)(x \cdot \eta)(\eta \cdot (x \times \xi)) + \\
\left\langle D^2 f_4(p) \xi, \xi \right\rangle (p \cdot \eta)^2 + 4 (Df_4(p) \cdot \xi)(p \cdot \eta)(\xi \cdot \eta) +f_4(p) (\xi \cdot \eta)^2 + \\
2\left\langle D^2 f_5(p) \xi, \xi \right\rangle (p \cdot \eta)(q \cdot \eta) + 4(Df_5(p) \cdot \xi)\left( (q \cdot \eta)(\xi \cdot \eta)+(p \cdot \eta)(\eta \cdot (x \times \xi)) \right) + \\
4f_5(p) (\xi \cdot \eta)(\eta \cdot (x \times \xi)) + \left\langle D^2 f_6(p) \xi, \xi \right\rangle (q \cdot \eta)^2 + 4(D f_6(p) \cdot \xi)(q \cdot \eta)(\eta \cdot (x \times \xi)) + \\
2 f_6(p) (\eta \cdot (x \times \xi))^2 + 2 \left\langle D^2 f_7(p) \xi, \xi \right\rangle(\hat{e} \cdot \eta)(x \cdot \eta) + 2 \left\langle D^2 f_8(p) \xi, \xi \right\rangle (\hat{e} \cdot \eta)(p \cdot \eta) + \\
4 (D f_8(p) \cdot \xi) (\hat{e} \cdot \eta)(\xi \cdot \eta) + 2 \left\langle D^2 f_9(p) \xi, \xi \right\rangle (\hat{e} \cdot \eta)(q \cdot \eta) + 4(D f_9(p) \cdot \xi)(\hat{e} \cdot \eta)(\eta \cdot (x \times \xi)) + \\
\left\langle D^2 f_{10}(p) \xi, \xi \right\rangle(\hat{e} \cdot \eta)^2 + \left\langle D^2 f_{11}(p) \xi, \xi \right\rangle \left\vert \eta \right\vert^2
\end{multline}

Since $\xi \cdot \eta = 0$ and $x \cdot \eta = x \cdot \xi = 0$, this reduces to:
\begin{multline}\label{eq:cscfull}
\sum_{i, j, k, l}D_{p_{i}p_{j}}\left( \nabla^{2}_{xx} c(x, y) \vert_{y = T} \right)_{kl} \xi_{i}\xi_j\eta_k \eta_l = 
\left\langle D^2 f_4(p) \xi, \xi \right\rangle (p \cdot \eta)^2 + \\
2\left\langle D^2 f_5(p) \xi, \xi \right\rangle (p \cdot \eta)(q \cdot \eta) + 4(Df_5(p) \cdot \xi)(p \cdot \eta)(\eta \cdot (x \times \xi)) + \\
\left\langle D^2 f_6(p) \xi, \xi \right\rangle (q \cdot \eta)^2 + 4(D f_6(p) \cdot \xi)(q \cdot \eta) (\eta \cdot (x \times \xi)) + \\
2 f_6(p) \left\vert \xi \right\vert^2 \left\vert \eta \right\vert^2 + 2 \left\langle D^2 f_8(p) \xi, \xi \right\rangle (\hat{e} \cdot \eta)(p \cdot \eta) + 2 \left\langle D^2 f_9(p) \xi, \xi \right\rangle (\hat{e} \cdot \eta)(q \cdot \eta) + \\
4(D f_9(p), \xi)(\hat{e} \cdot \eta)(\eta \cdot (x \times \xi)) + \left\langle D^2 f_{10}(p) \xi, \xi \right\rangle(\hat{e} \cdot \eta)^2 + \left\langle D^2 f_{11}(p) \xi, \xi \right\rangle \left\vert \eta \right\vert^2.
\end{multline}

\subsubsection{Checking the Cost-Sectional Curvature Condition for the Point-to-Point Cost Function}

\begin{theorem}
The point-to-point cost function does not satisfy the $\mathbf{Aw}$ condition.
\end{theorem}

\begin{proof}
If $x = -\hat{e}$, then $\hat{e} \cdot p = \hat{e} \cdot q = \hat{e} \cdot \eta = \hat{e} \cdot \xi = 0$. Also, by Equation~\eqref{eq:system1}, we have $y \cdot \hat{q} = 0$, so we must check the reduced necessary condition:

\begin{multline}\label{eq:cscreducedmore}
\sum_{i, j, k, l}D_{p_{i}p_{j}}\left( \nabla^{2}_{xx} c(x, y) \vert_{y = T} \right)_{kl} \xi_{i}\xi_j\eta_k \eta_l = \\
\left\langle D^2 f_4(p) \xi, \xi \right\rangle (p \cdot \eta)^2 + \left\langle D^2 f_{11}(p) \xi, \xi \right\rangle \left\vert \eta \right\vert^2 \leq -C_0 \left\vert \xi \right\vert^2 \left\vert \eta \right\vert^2.
\end{multline}

Since this should be true for all $p, \eta \in T_{x}$, the simplest necessary condition then is to check that $f_{11}(p) = -(x \cdot y) F_{1} - F_2$ is strictly concave as a function of $p$. Since $\hat{e} \cdot \hat{p} = \hat{e} \cdot \hat{q} = 0$, then $f_{11}(p)$ will be simply a function of $\left\Vert p \right\Vert$. Thus, we can check if $f_{11}(p)$ is concave with respect to $\left\Vert p \right\Vert$.









For $x \cdot \hat{e} = -1$, we get $y \cdot \hat{e} = -x \cdot y$ and $\hat{e} \cdot \hat{p} = \hat{e} \cdot \hat{q} = 0$. Therefore, by Equation~\eqref{eq:system1}, we get $y \cdot \hat{q} = 0$ and:

\begin{equation}
-\left\Vert p \right\Vert \left( \alpha(1+a)(1+a (x \cdot y))-1 + x \cdot y \right) = y \cdot \hat{p}.
\end{equation}

So, we get:
\begin{equation}
y \cdot \hat{p} = -\left\Vert p \right\Vert \left( \frac{1+a}{1-a} \right) \left( 1 + x \cdot y \right),
\end{equation}
which, by Equation~\eqref{eq:system1} again, leads to:
\begin{equation}
x \cdot y = \frac{1 - \left\Vert p \right\Vert^2 \left( \frac{1+a}{1-a} \right)^2}{1+\left\Vert p \right\Vert^2 \left( \frac{1+a}{1-a} \right)^2}.
\end{equation}

Therefore, by Equation~\eqref{eq:ffunctions}, we get:
\begin{equation}
f_{11}(p) = -\frac{\frac{1}{1+a} (x \cdot y + a)}{\frac{2}{1-a}(1+a x \cdot y)-1 + x \cdot y}.
\end{equation}

This function not concave. We visualize this for $a=0.2$ and $a=0.5$. This is plotted in Figure~\ref{fig:f11}.

\begin{figure}
\includegraphics[width=0.8\textwidth]{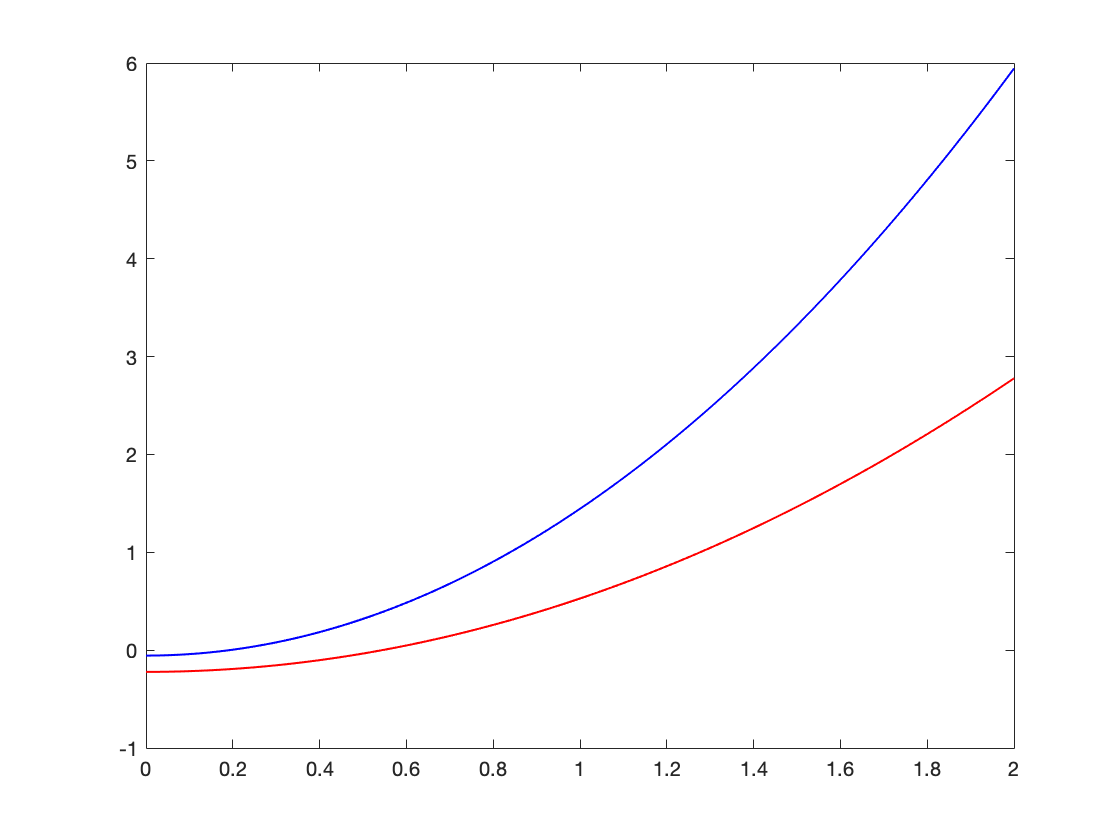}
\caption{The function $f_{11}$ plotted as a function of $\left\Vert p \right\Vert$ for $x \cdot \hat{e} = -1$ and $a=0.2$ (red) and $a = 0.5$ (blue).}\label{fig:f11}
\end{figure}

To show this is not concave, we compute $f_{11}(0)$. Denote $x \cdot y = \xi(\left\Vert p \right\Vert)$. Then, we compute that $\xi'(0) = 0$. Therefore, we get:
\begin{equation}
f_{11}''(0) = \frac{\xi''(0) \left( \frac{2a}{1-a}H(0) - \frac{1}{1+a}L(0) \right)}{L(0)^2},
\end{equation}
where $H(\xi) = \frac{a}{1+a} + \frac{\xi}{1+a}$ and $L(\xi) = \frac{2}{1-a}-1 + \xi \left( \frac{2a}{1-a}+1 \right)$. Therefore, $H(\xi(0)) = 1$ and $L(\xi(0)) = 2 \frac{1+a}{1-a}$. Thus, we compute:
\begin{equation}
f_{11}''(0) = 2.
\end{equation}
\end{proof}

This demonstrates that the point-to-point cost function does not satisfy the MTW condition $\textbf{Aw}$ and therefore, we cannot expect good regularity properties.

\section{Regularity}\label{sec:regularity}
The computations in Section~\ref{sec:verify} lead directly into this section, but here we state the ideas in generality. Denote $\tilde{p}_1$ to be the largest value of $\left\Vert p \right\Vert$ such that $[0, \tilde{p}_1] \times [-1, 1]^3 \subset \Omega_2$. Denote $\xi^{*}_{1}$ to be the largest value of $x \cdot y$ such that $[\xi^{*}, 1] \times [-1, 1]^3 \subset \Omega_1$ and $\xi^{*}_2$ to be the largest value of $x \cdot y$ such that $[\xi^{*}, 1] \times [-1, 1]^2 \subset \Omega$. Denote $\xi^{*} = \max \left\{\xi^{*}_{1}, \xi^{*}_{2} \right\}$. Then, we write $x \cdot y = \xi \left( \left\Vert p \right\Vert, x \cdot \hat{e}, \hat{e} \cdot \hat{p}, \hat{e} \cdot \hat{q} \right)$, we denote by $\tilde{p}_2$ the largest value of $\left\Vert p \right\Vert$ such that:
\begin{equation}
\xi^{*} \leq \min_{\xi_{2} \in [0, \left\Vert p \right\Vert] \times [-1,1]^3} \xi(\xi_2).
\end{equation}

Define $p^{*} = \min \left\{\tilde{p}_1, \tilde{p}_2 \right\}$. Now, we are ready to define our domain $D$ on which the MTW conditions will hold. Letting $\gamma$ be a constant that satisfies $0<\gamma<p^{*}$ and fix $x$. Define the map $S_{x}: \mathbb{S}^2 \rightarrow \mathbb{S}^2$ via $S_{x}(\hat{p}) = \arccos \left( \xi \left( \gamma, x \cdot \hat{e}, \hat{e} \cdot \hat{p}, \hat{e} \cdot (x \times \hat{p}) \right) \right)$ and the set $V_{x} = S_{x}(\mathbb{S}^2)$. We then define the set $D_{\gamma}$ via:
\begin{equation}\label{eq:dgamma}
D_{\gamma} = \cup_{x} \left\{ x \right\} \times V_{x}(\gamma).
\end{equation}

This choice of $D_{\gamma}$ is done such that, for a fixed $x$, the set $V_{x}$ is $c$-convex with respect to $x$, which is a necessary condition for regularity of solutions of the Optimal Transport PDE. We thus have the following theorem:

\begin{theorem}
Let the cost function with preferential direction satisfy Hypothesis~\ref{hyp:mainhypothesis} and also $\mathbf{As}$ on $D_{\gamma}$. Let $\mu$ and $\nu$ be two $C^{\infty}$ probability measures that are strictly bounded away from zero and $u$ a $c$-convex function such that $T(u)_{\#}\mu = \nu$ and also $\left\vert \nabla u \right\vert \leq \tilde{p}$. Then, $u \in C^{\infty}(\mathbb{S}^2)$.
\end{theorem}

\begin{proof}
We can utilize Theorem 19 of~\cite{T_LensRefractor} since $\nabla T(x)$ for $p$ satisfying $\left\Vert p(x) \right\Vert <p^{*}$, we can find positive $C^{\infty}$ probability measures $\mu$ and $\nu$ such that the magnitude of the gradient of the potential function is controlled. This allows us to ensure that, for a fixed $x$, the point $y \in V_{x}$. Cost functions satisfying Hypothesis~\ref{hyp:mainhypothesis} satisfy the MTW conditions on $D_{\gamma}$, except for the strictly negative cost sectional curvature. If this is additionally satisfied, then such cost functions satisfy all the MTW conditions on $D_{\gamma}$. As explained in Loeper~\cite{Loeper_OTonSphere}, this allows for the \textit{a priori} bound from~\cite{MTW} to be used and then the method of continuity employed to prove that positive $C^{\infty}$ density functions lead to a $C^{\infty}$ solution of the Optimal Transport PDE.
\end{proof}

Since we cannot expect good regularity properties for the solution of the point-to-point problem, since it does not satisfy condition $\mathbf{Aw}$, we can only expect to solve the Optimal Transport PDE~\eqref{eq:OTPDE} in a weak sense. Furthermore, if the potential function is not $C^1$, then the Optimal Transport mapping $T$ is not single valued, so we must slightly extend the definition of the Monge problem from Section~\ref{sec:introduction}. This extension of the meaning of the Monge problem can be thought of as a relaxation of the Monge problem and is a particular case of the Kantorovich formulation of Optimal Transport, see~\cite{Villani1} or~\cite{santambrogio} for excellent introductions to the subject. The most natural notion of weak solution $u$, for the Optimal Transport PDE, is the notion of a generalized solution. And, the most natural generalization of the Optimal Transport map $T$, will be a set-valued function $T_{u}$. Recall the definition of the Optimal Transport functional $\mathcal{C}(S)$ from Section~\ref{sec:introduction}. This is the primal formulation of the Monge Optimal Transport problem, where we solve for the Optimal Transport mapping $T$ that minimizes $\mathcal{C}(S)$. The dual formulation is to find a pair of functions $(u,v)$ that maximize the following functional:

\begin{equation}
K(\varphi, \psi) = \int_{\mathbb{S}^2} f(x) \varphi(x) dx + \int_{\mathbb{S}^2} g(y) \psi(y) dy,
\end{equation}
subject to the constraint $\varphi(x) + \psi(y) \leq c(x,y), \forall x \in \mathbb{S}^2, y \in \mathbb{S}^2$. It is a classical result, see~\cite{MTW} for an explanation, for example, that the maximizers of $K(\varphi, \psi)$ subject to the constraint are $c$-convex functions $(u,v)$, which then implies that they are semi-convex. Semi-convex functions are continuous up to the boundary. Moreover, their second derivatives exist almost everywhere and, consequently, are Lipschitz continuous on $\mathbb{S}^2$. Then, defining the set-valued function (sometimes referred to as the ``contact set"):
\begin{equation}
T_{\varphi}(x) := \left\{ y \in \mathbb{S}^2: \varphi(x) + \varphi^{c}(y) = -c(x,y) \right\},
\end{equation}
where the superscript ``c" denotes the $c$-transform defined in Section~\ref{sec:introduction}, we may define what we mean by the generalized solution of the Optimal Transport PDE~\eqref{eq:OTPDE}, where $u$ is Lipschitz continuous, and $T_{u}$ is set valued.

\begin{definition}\label{eq:Tdef}
A generalized solution of the Optimal Transport PDE~\eqref{eq:OTPDE} is a $c$-convex function $u$ such that for any Borel set $E \subset \mathbb{S}^2$, we have:
\begin{equation}
\int_{E} f = \int_{T_{u}(E)}g,
\end{equation}
where $T_{u}$ is defined by Equation~\eqref{eq:Tdef}.
\end{definition}

To reiterate, the generalized solution of the Optimal Transport problem is the solution of the Kantorovich Optimal Transport problem. Since we cannot expect smoothness for the point-to-point problem even for $C^{\infty}$ density functions $f$ and $g$, the best we can do is contrive conditions on the source and target distributions $\mu$ and $\nu$ such that they do not require mass to move a distance greater than $\arccos \xi^{*}$. We choose to confine them to a small ball of radius $\arccos (\xi^{*})/2$.

\begin{theorem}
Let $\mu$ and $\nu$ be probability measures with $L^1$ density functions $f$ and $g$, respectively, such that $\int_{\mathbb{S}^2} f = \int_{\mathbb{S}^2}g = 1$ and fixed $x_{0} \in \mathbb{S}^2$ such that $\text{supp}(\mu) \subset B_{x_{0}}(\arccos (\xi^{*})/2)$ and $\text{supp}(\nu) \subset B_{x_{0}}(\arccos (\xi^{*})/2)$. Then, there exists a unique Lipschitz continuous solution (generalized solution), up to a constant, to the Optimal Transport PDE~\eqref{eq:OTPDE}.
\end{theorem}

\begin{proof}
Since $\mu$ and $\nu$ are contained in the ball $B_{x_{0}}(\arccos(\xi^{*})/2)$, then we apply the MTW conditions $\mathbf{A1}$ and $\mathbf{A2}$ on $D_{\gamma}$ to conclude that there exist maximizers of the Kantorovich dual formulation of the Optimal Transport problem. Such maximizers are necessarily $c$-convex functions, which implies that they are Lipschitz continuous. By~\cite{MTW}, such maximizers are generalized solutions of the Optimal Transport PDE~\eqref{eq:OTPDE}.
\end{proof}

\section{Conclusion}\label{sec:conclusion}
We explored the theory of cost functions with preferential direction, with an aim to answer questions about the existence, uniqueness, and regularity of solutions of the Optimal Transport PDE for the point-to-point cost function. We defined sufficient hypotheses on the cost function that guaranteed most of the MTW conditions ($\mathbf{A0}$, $\mathbf{A1}$ and $\mathbf{A2}$) on a domain $D$. This thereby guaranteed the existence of unique (up to a constant) Lipschitz continuous generalized solutions of the Optimal Transport PDE, provided that the source and target density functions could control how far the mapping moved mass. We also derived the negative cost-sectional curvature condition for cost functions with preferential direction, the final MTW condition $\mathbf{Aw}$ and $\mathbf{As}$. Using these formulas and hypotheses, we were able to get regularity guarantees for a wide class of cost functions with preferential direction, provided that mass was not required to move too far. We could also show that provided that the source and target density functions had support in a certain set, the Optimal Transport PDE with the point-to-point cost function had a unique Lipschitz continuous solution up to a constant.

\textbf{Acknowledgements}: I would like to thank Brittany Hamfeldt for introducing me to these optics problems and working alongside for some of the initial exploratory computations.

\bibliographystyle{plain}
\bibliography{ThreeSystems}

\end{document}